\documentclass[11 pt]{amsart}

\usepackage{times}
\usepackage{geometry}
\usepackage{amssymb}
\usepackage{latexsym,amssymb,  amsmath, amscd, amsfonts}
\usepackage{graphicx}
\usepackage[percent]{overpic}
\usepackage{pdfsync}
\usepackage{units}
\usepackage{hyperref}
\usepackage{euscript}
\usepackage{multicol}
\usepackage{epstopdf}
\usepackage{paralist}

\newtheorem{theorem}{Theorem}
\newtheorem{lemma}[theorem]{Lemma}
\newtheorem{proposition}[theorem]{Proposition}

\theoremstyle{definition}
\newtheorem{definition}[theorem]{Definition}

\newtheorem{const}[theorem]{Construction}
\newtheorem{remark}[theorem]{Remark}

\newcommand{\R}{\mathbb{R}}     
\newcommand{\Z}{\mathbb{Z}}       

\def\Rib{{\operatorname{Rib}}}
\def\Cr{{\operatorname{Cr}}}
\def\Len{{\operatorname{Len}}}

\begin{document}

\title[Pretzel links and ribbonlength bounds] {Ribbonlength bounds for pretzel links and \\ knots with $\leq 9$ crossings}
\author{Elizabeth Denne}
\address{Elizabeth Denne: Washington \& Lee University, Department of Mathematics, Lexington VA}
\email{dennee@wlu.edu}
\urladdr{https://elizabethdenne.academic.wlu.edu/}
\date{\today}
\makeatletter								
\@namedef{subjclassname@2020}{%
  \textup{2020} Mathematics Subject Classification}
\makeatother


\begin{abstract}
Given a thin strip of paper, tie a knot, connect the ends, and flatten into the plane. This is a physical model of a folded ribbon knot in the plane, first introduced by Louis Kauffman. We study the folded ribbonlength of these folded ribbon knots, which is defined as the knot's length-to-width ratio. The {\em ribbonlength problem} asks to find the infimal folded ribbonlength of a knot or link type. We prove that any $P(p,q,r)$ pretzel link can be constructed so that its infimal folded ribbonlength is $\leq  \frac{55}{\sqrt{3}}  \leq 31.755$. We prove that any $n$-strand pretzel link $P(p_1,p_2, \dots, p_n)$ can be constructed so that its infimal folded ribbonlength  is $\leq \frac{18n+1}{\sqrt{3}}$. This means that there is an infinite link family with a uniform bound on infimal folded ribbonlength. That is, we have shown $\alpha=0$ in the equation $c\cdot \Cr(L)^\alpha \leq \Rib([L])$, where $L$ is any link and $c$ is a constant. This paper also contains a table showing the best known upper bounds on the infimal folded ribbonlength for all knots with $\leq 9$ crossings.

\end{abstract}

\maketitle

\section{Introduction to Folded Ribbon Knots and Links}\label{sect:intro}
A {\em knot} is an embedding of the unit circle $S^1$ into $\R^3$, and a {\em link} is a disjoint union of a finite number of knots.  Thus a knot is a one component link. Physical knots encountered in our everyday world can be modeled mathematically. For example, knot theorists have studied the {\em ropelength problem}, which asks to find the minimum-length configuration of a knotted diameter-one tube embedded in Euclidean three-space \cite{BS99, CKS, DEY, DEZ, GM}.   In this paper we look at the {\em folded ribbonlength problem}, which is a discrete two-dimensional analogue of the ropelength problem. This problem first appeared in recreational mathematics \cite{John,Wel}, where trefoil knots are tied in long rectangular strips of paper, then pulled ``tight'' to form a pentagonal shape. In 2005, the formal mathematical concept of a {\em folded ribbon knot} was introduced by L. Kauffman \cite{Kauf05}.  Figure~\ref{fig:tref-pent-stick} shows a folded ribbon trefoil knot in the center, while on the right we see the ``tight'' pentagonal shape well known to recreational mathematicians. 

\begin{center}
    \begin{figure}[htbp]
        \begin{overpic}{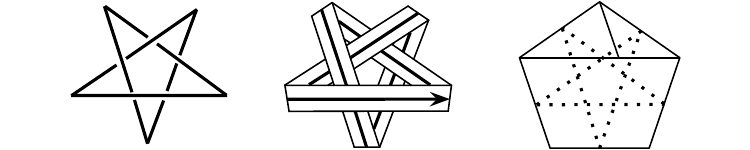}
        \put(36,6){$p$}
        \put(60.5,6){$q$}
        \end{overpic}
        \caption{On the left, a polygonal knot diagram for a trefoil knot. In the center, the corresponding folded ribbon trefoil knot. On the right, the folded ribbon trefoil has been ``pulled tight" to give a pentagonal shape. }
        \label{fig:tref-pent-stick}
    \end{figure}
\end{center}

Before we give a formal definition of a folded ribbon knot, we remind readers that a {\em polygonal knot diagram} is the image of a piecewise linear immersion of $K:S^1\rightarrow\R^2$ with consistent crossing information. We use $K$ to denote both the map and its image in $\R^2$. We denote the finite number of vertices in the polygonal diagram by $v_1, v_2, \dots, v_n$ and the edges connecting consecutive vertices by $e_1=[v_1,v_2],\dots, e_n=[v_n,v_1]$. If the knot is oriented, we assume the labeling of the vertices follows the orientation. 

When we consider folded ribbon knots (defined formally below), we see that near a fold line there is a choice of which  piece of ribbon lies over the other. We say there is an {\em overfold} at vertex $v_i$ if the ribbon corresponding to edge $e_i$ lies over the ribbon of edge $e_{i-1}$.  Similarly, an {\em underfold} if the ribbon corresponding to edge $e_i$ lies under the ribbon corresponding to edge $e_{i-1}$. In the center image in Figure~\ref{fig:tref-pent-stick}, there is an overfold at vertex $p$ and an underfold at vertex $q$. This choice of over or underfold at each fold line of $K_w$ is called the {\em folding information}.  We construct a folded ribbon link from a polygonal link diagram as follows (see also \cite{Den-FRS,DKTZ}).

\begin{definition}\label{def:FR} Given an oriented polygonal knot diagram $K$,  we define the {\em fold angle} $\theta_i$ at vertex $v_i$ (where $0\leq \theta_i\leq \pi$) to be the angle between edges $e_{i-1}$ and $e_i$. The {\em oriented folded ribbon knot of width $w$}, denoted $K_w$ is constructed as follows.
\begin{compactenum}
\item At each vertex we place a fold line that acts like a mirror segment. That is at each vertex $v_i$, if the fold angle $\theta_i<\pi$, place a fold line of length $w/\cos(\frac{\theta_i}{2})$ centered at $v_i$ perpendicular to the angle bisector of $\theta_i$. If $\theta_i=\pi$ at vertex $v_i$, there is no fold line.
\item Add in the ribbon boundary lines parallel to each edge $e_i$ by joining the ends of the fold lines at $v_i$. By construction, each boundary line is parallel to and at distance $w/2$ from $K$.
\item The ribbon has no singularities (is immersed) except at the fold lines, which are assumed to be disjoint.
\item $K_w$ has consistent crossing information, which agrees with the crossing information given by the knot diagram $K$ and with the folding information at each vertex.
\item The ribbon inherits the orientation from $K$.
\end{compactenum}

A {\em folded ribbon link $L_w$} is defined similarly, as each component is a folded ribbon knot.
\end{definition}

We observe that for a given polygonal knot diagram, if the width is too large, then the fold lines of consecutive vertices may touch. In \cite{Den-FRF}, we prove that every polygonal link diagram has a corresponding folded ribbon link for a small enough width. 

A natural question is to try to work out the least amount of ribbon needed to tie any folded ribbon link. This is the folded ribbonlength problem. To model this mathematically, we need a scale-invariant quantity, called  the folded ribbonlength.

\begin{definition} The {\em folded ribbonlength} $\Rib(L_w)$ of a folded ribbon link $L_w$ is the quotient of the length of the link diagram $L$ to the width $w$ of the ribbon, that is $\Rib(L_w) =\frac{\Len(L)}{w}$. The {\em (infimal) folded ribbonlength} of a link type $[L]$ is defined to be 
$$ \Rib([L])= \inf_{L_w\in[L]}\Rib(L_w).$$
\end{definition}
To simplify computations, we often assume that our folded ribbon link has width $w=1$. The folded ribbonlength is then just the length of the link. Observe that we just consider the link type when finding the infimal folded ribbonlength for a link. We do not consider the topological type of the ribbon (M\"obius band or annulus), nor do we consider the framing of the ribbon around the knot. 

When computing folded ribbonlength it is helpful to use the local geometry and trigonometry of folds and crossings --- complete details can be found in \cite{Den-FRF}.  As an example of this, we observe that locally, a fold in a piece of ribbon consists of two overlapping triangles, while a crossing consists of two overlapping parallelograms. The ribbonlength of any square ribbon segment is $1$. If we fold a piece of ribbon so that the fold angle $\theta = \nicefrac{\pi}{2}$, we get two overlapping isosceles right triangles, each with a ribbonlength of $\nicefrac{1}{2}$. This $\nicefrac{\pi}{2}$-fold has a total ribbonlength of $1$.

In order to find upper bounds on the infimal folded ribbonlength of a link type, one simply needs to construct a folded ribbon link and compute the folded ribbonlength.  Many authors \cite{Den-FRF, Den-TP, Kauf05, KMRT, KNY-TwTorus, KNY-2Bridge,Tian-A} have worked on this strategy over the years, often finding folded ribbonlength of infinite families of knots and links.  In addition, many of these papers give the upper bounds in terms of the crossing number\footnote{ The crossing number $\Cr(L)$ of a link $L$ is a knot invariant which is defined to be the minimum number of crossings in any regular projection of the link.} $\Cr(L)$ of the link $L$.  

The relationship between the infimal folded ribbonlength of a knot type $[K]$ and its crossing number $\Cr(K)$ is a well-known problem in the study of folded ribbon knots. The {\em ribbonlength crossing number problem} seeks to establish constants \( c_1 \), \( c_2 \), \( \alpha \), and \( \beta \) such that the following inequality holds for all knot and link types:
\begin{equation}
c_1\cdot \Cr(L)^\alpha \leq \Rib([L]) \leq c_2\cdot \Cr(L)^\beta.
\label{eq:crossing}
\end{equation}
Y. Diao and R. Kusner \cite{Kauf05} conjectured that $\alpha=1/2$ and $\beta=1$, and G. Tian \cite{Tian-A} proved $\beta\leq 2$, while the author \cite{Den-FRC} proved that $\alpha\leq 1/2$ and $\beta\leq 3/2$. However, we now know that the ribbonlength crossing number problem has recently been completely solved!   In 2024, H.~Kim, S.~No, and H.~Yoo~\cite{KNY-Lin} gave a solution for the upper bound in Equation~\ref{eq:crossing} by proving that that $\beta\leq 1$. Specifically, they proved that for any knot or link $L$, 
\begin{equation}\Rib([L])\leq 2.5\Cr(L)+1.
\label{eq:bound}
\end{equation}
In 2025, the author and T. Patterson \cite{Den-TP} gave a solution for the lower bound in Equation~\ref{eq:crossing} for knots by proving that $\alpha=0$. As will be seen below (in (4) and (6)), we proved that the $(2,q)$-torus knots and the twist knots have a uniform upper bound for folded ribbonlength, while the crossing numbers are unbounded. 

Here is a summary of the best upper bounds of folded ribbonlength to date. (Only a few results have been given in terms of crossing number.)

\begin{enumerate}
\item Any link $L$ has $\Rib([L])\leq 2.5\Cr(L)+1$ from \cite{KNY-Lin}.
\item Any trefoil knot $K$ has $\Rib([K])\leq 6$ from \cite{CDPP,Den-FRF}.
\item Any $(p,q)$-torus link $L$ has $\Rib([L])\leq 2p$ where $p\geq q\geq 2$ from \cite{Den-FRF}. This means if $p=aq+b$ for some $a,b\in\Z_{\geq 0}$, then $\Rib([L]) \leq \sqrt{6}(a+\frac{b}{3})(Cr(L))^{1/2}$. 

\item Any $(2,q)$-torus link has 
$$\Rib([T(2,q)])\leq\begin{cases} 8\sqrt{3}\leq 13.86 & \text{for $q$ odd from \cite{Den-TP},}\\
q+3 & \text{ for any $q$ from \cite{CDPP}}.
\end{cases}
$$
\item Any figure-8 knot $K$ has $\Rib([K])\leq 8$ from \cite{CDPP}.
\item Any twist knot $T_n$, with $n$ half-twists has 
$$\Rib([T_n])\leq \begin{cases} 9\sqrt{3}+2 \leq 17.59 & \text{ when $n$ is odd from \cite{Den-TP},} \\
8\sqrt{3} +2 \leq  15.86 & \text{ when $n$ is even from \cite{Den-TP},}\\
n+6  \quad & \text{ for $n\leq 9$ and $11$ from \cite{CDPP}}.
 \end{cases}$$
 \item Any $(p,q,r)$-pretzel link $L$ has $\Rib([L])\leq (|p|+|q|+|r|) + 6$ from \cite{CDPP}.
\item Any $n$-strand pretzel link $L$ constructed from $(p_1,p_2,\dots,p_n)$ half-twists, has $\Rib([L])\leq(\sum_{i=1}^n|p_i|)+2n$ from \cite{CDPP}.
\item Any $2$-bridge knot with crossing number $\Cr(K)$ has $\Rib([K])\leq  2\Cr(K)+2$ from \cite{KNY-2Bridge}.
\item  Any twisted torus knot\footnote{A twisted torus knot $T_{p,q;r,s}$ is obtained from a torus knot $T_{p,q}$  by twisting $r$ adjacent strands $s$ full twists.} $K=T_{p,q;r,s}$ (from \cite{KNY-TwTorus}) has
$$\Rib([K]) \leq \begin{cases} 2(\max\{p,q,r\}+|s|r) \quad \text{for $0<r<p+q$ and $s\neq 0$,}
\\ 2(p+(|s|-1)r) \quad \text{for $0< r\leq p-q$ and $s\neq 0$.}
\end{cases}
$$
\end{enumerate}

We see that many of these results also show that $\beta=1$ in Equation~\ref{eq:crossing} for infinite families of knots and links. Moreover, they improve the upper bound given in Equation~\ref{eq:bound}. We again remark that the $(2,q)$ torus knots and the twist knots both have a uniform upper bound on folded ribbonlength. This surprising result shows that the ribbonlength bound does not depend on the complexity of these knots as captured by the crossing number. 

The primary aim of this paper is to prove that pretzel links\footnote{See Definition~\ref{def:pretzel}.} also have a uniform upper bound on folded ribbonlength, independent of the number of half-twists used in their construction.  The key result needed for this proof, is that we can construct any number of half-twists between two pieces of ribbon in a finite amount of ribbonlength (see Construction~\ref{const:twists}).

\begin{theorem} \label{thm:twists} Assume $k$ is a nonzero integer, and assume that we have folded $k$ half twists between two pieces of ribbon of the same width. Then, ignoring the ends of the ribbons, these $k$ half-twists can be constructed with folded ribbonlength
$$\Rib(k \text{ half-twists})=\frac{12}{\sqrt{3}}+ \epsilon, \text{ for all $\epsilon >0$}.$$
\end{theorem}

In order to construct a pretzel link, we only need a finite amount of ribbonlength to join the half-twists together. Thus we can prove the following. 
\begin{theorem}\label{thm:pretzel} The infimal folded ribbonlength of any $P(p,q,r)$ pretzel link is
$$\Rib([P(p,q,r)])\leq \begin{cases} \frac{55}{\sqrt{3}}  \leq 31.755   & \text{ when $p$, $q$, $r$ have the same parity},
\\  \frac{49}{\sqrt{3}}  \leq  28.291 & \text{ when one of $p$, $q$, $r$ has the opposite parity to the others}.
\end{cases}
$$
\end{theorem}

It turns out that we can generalize the construction for a folded ribbon $P(p,q,r)$ pretzel link, to a construction for a folded ribbon $n$-strand pretzel link $P(p_1,p_2,\dots ,p_n)$. Here, the upper bound most likely will not be close to the actual infimum. 

\begin{theorem}\label{thm:pretzel-mi} The infimal folded ribbonlength of any $n$-strand pretzel link type $P(p_1,p_2,\dots p_n)$ is 
$$ \Rib([P(p_1,p_2,\dots p_n)]) \leq \frac{18n+1}{\sqrt{3}}.$$
\end{theorem}

Intuitively, we expect the amount of material we use to tie a physical link to increase, as the link's complexity increases. Thus finding there is a bounded amount of folded ribbonlength is surprising.
Note that the pretzel link family contains both knots and links. This means we can now prove the lower bound in Equation~\ref{eq:crossing} for {\em links} (previously this was only proven for knots).  Simply by observing that the crossing number of pretzel links can be arbitrarily large while the folded ribbonlength is constant, we have proved the following. 

\begin{theorem}\label{thm:pretzel-bound} Given any link $L$ and positive constant $c_1$, the equation $c_1\cdot\Cr(L)^\alpha\leq \Rib([L])$, must have $\alpha=0$.
\end{theorem}

As discussed above, there have been many papers giving upper bounds on infimal ribbonlength of particular knot and link families. These bounds have not been gathered into one place and applied to small crossing knots. Thus, the second aim of this paper is to give a table showing the best known upper bounds on folded ribbonlength for each knot type. This summary is for knots with crossing number $\leq 9$ and several link types.  These tables are found in Section~\ref{sect:table}. 

We end our introductory remarks by observing that the ribbon gives a framing around the link diagram and we can study this in addition to the link type when considering folded ribbonlength. However, we will not pursue this approach here. Instead we note that there have been a number of results about framed unknots in \cite{RES-Brown, Den-FRLU, Den-TP, Hen,RES-Mont, RES-Mob2, RES-Mob}. In some of these citations, the reader will see that folded ribbon unknots have been used to give results about infimal aspect ratio of embedded multi-twist M\"obius bands and annuli. This is a new development given the recent results of R.E.~ Schwartz and his coauthors. We close by noting that the framed nontrivial knot case remains open.

The paper is arranged as follows. In Section~\ref{sect:pretzel}, we review the definition of a half-twist in two arcs of a knot. These are the building blocks of pretzel links, which we define in Definition~\ref{def:pretzel}. We also review several properties of pretzel links needed later. In Section~\ref{sect:twists} Definition~\ref{def:folds}, we introduce the key idea of an accordion fold and then describe the geometry and distances of these folds. In Construction~\ref{const:twists}, we use accordion folds to construct $k$ half-twists in two equal width pieces of ribbon. We then compute the folded ribbonlength of Construction~\ref{const:twists} to prove Theorem~\ref{thm:twists}. We end Section~\ref{sect:twists} with two technical lemmas giving the folded ribbonlength of certain folding patterns that will be used when constructing pretzel links. In Section~\ref{sect:ribbon}, we construct $P(p,q,r)$ pretzel links using the half-twists from Construction~\ref{const:twists}. We split the construction into two parts, Construction~\ref{const:pretzel-odd} is when $p$, $q$, $r$ have the same parity, and  Construction~\ref{const:pretzel-parity} is when one of $p$, $q$, $r$ is of opposite parity to the other two. We prove Theorem~\ref{thm:pretzel} by computing the folded ribbonlength of these constructions. We then generalize these constructions to $n$-strand pretzel links to prove Theorem~\ref{thm:pretzel-mi}. Finally, in Section~\ref{sect:table}, we give several tables showing the best known upper bound on the infimal folded ribbonlength for all knots with $\leq 9$ crossings.

We encourage the reader to locate some long strips of paper with which they can recreate the constructions found in this paper. The constructions will make more sense with a physical model in hand.

\section{Half-twists and pretzel links} \label{sect:pretzel}

One of the defining characteristics of pretzel links is that they are made from strands of half-twists that are joined together in a specific way.  A {\em strand of $k$ half-twists} is constructed by taking two parallel straight arcs and rotating the bottom ends of the arcs by $\pi$ radians. If $k$ is positive, then ends of the arcs are twisted in a counterclockwise direction and if $k$ is negative, the ends of the arcs are twisted in a clockwise direction. There is a choice here, and other texts may use the opposite convention.  In knot diagrams, a strand of $k$ half-twists is sometimes replaced by a box with the integer $k$. This is shown on the left in Figure~\ref{fig:pretzel-box}, while on the right we see three strands of differently oriented half-twists joined to give a $P(3,1, -2)$ pretzel knot.

\begin{center}
    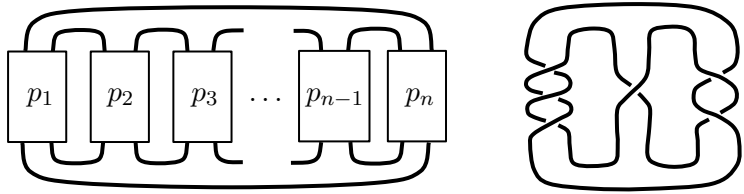
\begin{figure}[htbp]
        \begin{overpic}{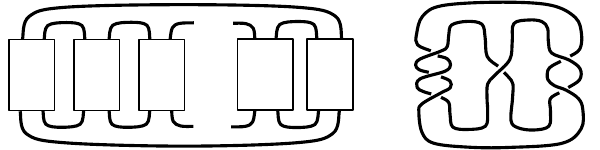}
        \put(4,12){$p_1$}
        \put(14.5,12){$p_2$}
        \put(25.5,12){$p_3$}
        \put(33,12){$\dots$}
        \put(40.5,12){{$p_{n-1}$}}
        \put(53.5,12){$p_n$}
        \end{overpic}
        \caption{On the left, an $n$ strand $(p_1,p_2,\dots, p_n)$-pretzel link, and on the right a $P(3, 1, -2)$ pretzel knot.}
        \label{fig:pretzel-box}
    \end{figure}
\end{center}

\begin{definition}\label{def:pretzel} An {\em $n$-strand pretzel link} has $n$ strands of half-twists $(p_1, p_2, \dots , p_n)$ which are connected cyclically, as shown in Figure~\ref{fig:pretzel-box}.  A 3-strand pretzel link is usually just called a pretzel link, and is normally denoted by $P(p,q,r)$.
\end{definition}

\begin{remark}\label{rmk:pretzel}
We review some well-known facts about pretzel links (see for instance \cite{Crom, JohnHen}).
\begin{compactenum}
\item A pretzel link $P(p,q,r)$ is a knot if and only if at most one of $p$, $q$ and $r$ are even. 
\item A pretzel link $P(p,q,r)$ is alternating if and only if $p,q$ and $r$ have the same sign.
\item The mirror image $P^m(p,q,r)$ of a pretzel link is equivalent to $P(-p,-q,-r)$.
\item A pretzel link $P(p,q,r)$ is equivalent to a pretzel link with any permutation of $p$, $q$ and $r$. For example, $P(p,q,r) \cong P(p,r,q) \cong P(q,r,p)$.
\item The crossing number of a pretzel link can be hard to find in general. We have the following specific cases:
\begin{compactenum} 
\item If $p,q,r$ have the same sign, then $P(p,q,r)$ has a reduced alternating diagram, so $\Cr(P(p,q,r))=|p|+|q|+|r|$.
\item  The crossing number of a pretzel link $P(-p,q,r)$ where  $p,q,r\ge2$ is given by $\Cr(P(-p,q,r))=p+q+r$ (see \cite{LeeJin}).
\end{compactenum}
\end{compactenum}
\end{remark}


\section{Accordion folds and half-twists} 
\label{sect:twists}
In this section, we show how to fold a strand of $k$ half-twists out of two pieces of ribbon in an efficient way, and compute how much folded ribbonlength is needed for this construction.  In particular, we introduce the notation needed to describe folded ribbon half-twists, as well as several technical results used in the construction of a pretzel link. 

To get started, we first change our perspective on half-twists by straightening one of the arcs, and imagining the other arc as wrapping around the first. See Figure~\ref{fig:half-twists-new}. We now imagine creating these $k$ half-twists out of pieces of ribbon.  If we could make the ribbon corresponding to the straight strand very narrow, then we could then coil a second ribbon around the first in a way which builds the folded ribbon $k$ half-twists in a finite amount of ribbon.  There are two technical issues here: (a) folding a narrow ribbon and (b) working out how to locate the end of the ribbon inside the coil. 

\begin{center}
\begin{figure}[htbp]
\begin{overpic}{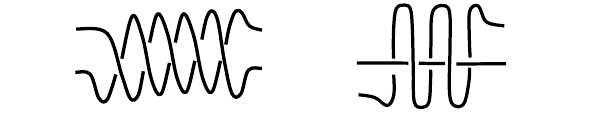}
\put(9,5.5){$A$}
\put(44,13){$B$}
\put(9.5,13){$C$}
\put(44,6){$D$}
\put(56.5,2){$A$}
\put(84,13){$B$}
\put(56,7){$C$}
\put(85,7){$D$}
\end{overpic}
\caption{Two different viewpoints of $5$ half-twists.}
\label{fig:half-twists-new}
\end{figure}
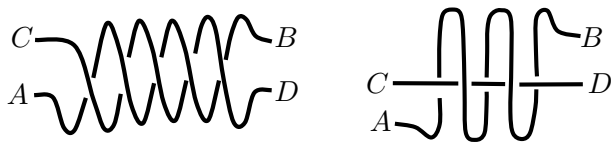
\end{center}

The solution is to take the strip of ribbon and fold it back and forward like an accordion. This firstly creates a narrow strip and secondly, creates a way to reach the inside end of a tightly coiled ribbon. In 2024, A. Hennessey introduced these ideas in  \cite{Hen}, and they were further developed by the author and Patterson in \cite{Den-TP}.

%

\begin{definition}[Accordion fold] \label{def:folds} 
Take an oriented piece of  ribbon. As shown in Figure~\ref{fig:accord-details}, an {\em accordion fold with fold angle $\theta$ and distance $d$} is an alternating sequence of left then right overfolds, each with fold angle $\theta$ and at equal distance to one another. Thus in an accordion fold, there is a constant distance $d$ between each vertex of the corresponding knot diagram.

An {\em escape accordion} is an accordion fold with an even number of folds, say $k=2m$ of them, such that the parallel pieces of ribbon entering and leaving the accordion fold are at least distance $w=1$ apart. 
\end{definition}

\begin{center}
\begin{figure}[htpb]
\begin{overpic}{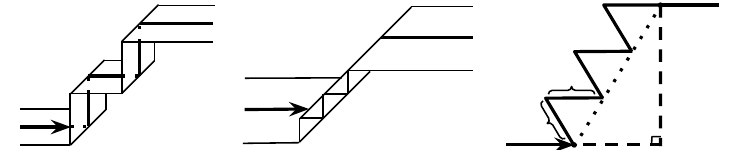}
\put(75.5,-1.5){$v_S$}
\put(87,20.5){$v_E$}
\put(88.5,-0.5){$P$}
\put(71,2){$d$}
\put(74.5,4.75){\footnotesize{$\pi/3$}}
\put(75,9){$d$}
\end{overpic}
\caption{On the left, a sequence of accordion folds with fold angle $\pi/2$. In the center, an escape accordion.  On the right, the path of the knot diagram in an accordion fold with fold angle $\pi/3$. The dotted lines in this figure helps us compute the distances connected to an accordion fold.}
\label{fig:accord-details}
\end{figure}
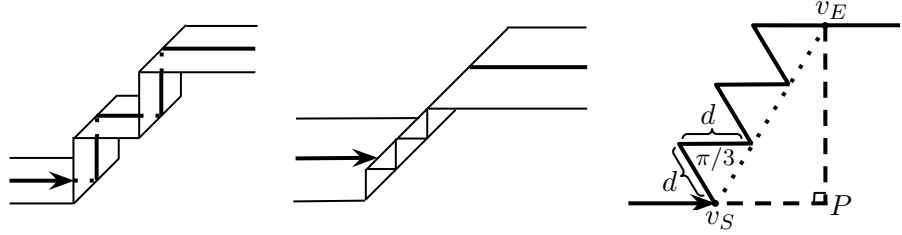
\end{center}

\begin{remark} Note that we could have chosen to construct accordion folds using only underfolds or by folding to the right at the start. {\bf In this paper we assume all accordion folds start with a left overfold.}  We denote the start and end vertices of as escape accordion by $v_S$ and $v_E$ respectively. We denote the distance between the start and end points of the escape accordion along the knot by $d_K(v_S, v_E)$. We denote the planar distance between the start and end points of the escape accordion by $d(v_S, v_E)$.  

In Figure~\ref{fig:accord-details} right, we assume that line segment $v_EP$ is perpendicular to the line segment containing $v_SP$.  Thus for an {\em escape accordion}, we require the distance $d(v_E,P)$ to be at least the width $w=1$. Note that given a particular number of folds $k$, the distance $d$ can be adjusted so that $d(v_E,P)=w=1$.
\end{remark}

Once we have constructed an escape accordion, we can continue to add accordion folds, or we can wrap the ribbon around the escape accordion as shown in Figure~\ref{fig:accord-notation}. 

\begin{definition}[$k$ half-wraps] \label{defn:wrap} Assume we have an escape accordion with fold angle $\theta$ and distance $d$. We construct {\em $+k$ half-wraps} as follows. Start at vertex $v_E$, and fold an alternating sequence of a left underfold followed by a right overfold, each with the same fold angle $\theta$ and the same constant distance $d$ between each vertex of the corresponding knot diagram. At each stage, the ribbon winds around the existing accordion folds. Without loss of generality, we declare these to be positive half-wraps, and $+3$ half-wraps are shown in in Figure~\ref{fig:accord-notation}. In the positive case, the end of the ribbon of $+(2n+1)$ half-wraps finishes behind the accordion fold, and the end of ribbon of the $+2n$ half-wraps finishes in front of the accordion fold. 

To fold negative half-wraps, we start with an escape accordion. This time at vertex $v_E$ we fold an alternating sequence of a left overfold followed by a right underfold. At each stage, the ribbon winds around the existing accordion folds. In the negative case, the end of the ribbon of $-(2n+1)$ half-wraps finishes in front of the accordion fold, and the end of the ribbon of $-2n$ half-wraps finishes behind the accordion fold. 
\end{definition}

\begin{center}
    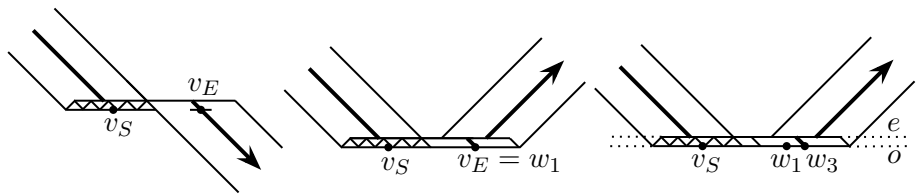
\begin{figure}[htpb]
    \begin{overpic}{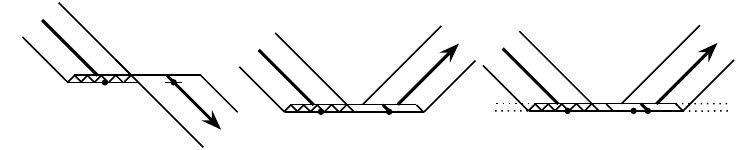}
    \put(13,6.75){$v_S$}
    \put(21.75,11){$v_E$}
    \put(42,2.75){$v_S$}
    \put(50,3){$v_E=w_1$}
    \put(74.5,3){$v_S$}
    \put(82.75,3){$w_1$}
    \put(86.5,3){$w_3$}
    \put(95,3.5){$o$}
    \put(95,6.75){$e$}
    \end{overpic}
    \caption{On the left, an escape accordion with labels. In the center, an escape accordion with $+1$ half-wrap with vertex $v_E=w_1$. On the right, an escape accordion with $+3$  labeled half-wraps.}
    \label{fig:accord-notation}
    \end{figure}
\end{center}

\begin{remark}\label{rmk:notation} We first observe that the knot diagram is the same for a sequence of half-wraps as it is for a sequence of accordion folds.  Secondly, to help keep track of the distances, we label the half-wraps by the vertex of the knot diagram at the start of the half-wrap. Thus the end vertex of the escape accordion is the same vertex as the first wrap, or $v_E=w_1$. We add in labels $w_2, w_3, \dots, w_{2n+1}$ as we add in each half-wrap as  shown in Figure~\ref{fig:accord-notation}.  Thirdly, we also need to keep track of accordion folds after the end vertex $v_E$ of the escape accordion. Thus, we will use the $w_i$ notation to indicate the start of the $i$th accordion fold after the vertex $v_E$. Finally, observe that vertex $w_{2n+1}$ is both the end point of the $2n$th wrap/accordion fold and the start point of the $(2n+1)$th wrap/accordion fold. 
\end{remark}

We need a little more notation which we will use in our constructions later. 

\begin{definition} Assume we have a piece of ribbon with an escape accordion with fold angle $\theta$ and distance $d$. Assume there are additional accordion folds or half-wraps after the escape accordion.  Then the line containing the vertices $v_S$, $v_E$, and odd vertices $w_{2n+1}$ is the {\em odd fold line}, denoted by $o$. The line containing the even vertices $w_2$, \dots, $w_{2n}$ is the {\em even fold line}, denoted by $e$. These lines are shown on the right in Figure~\ref{fig:accord-notation}.
\end{definition}

From this point on we will assume that all of our accordion folds, escape accordions, and half-wraps are made with fold angle $\pi/3$ and distance $d$.  We make this choice, as we have previously shown in~\cite{Den-TP} that an accordion fold constructed with fold angle $\pi/3$ has the smallest amount of ribbonlength.  We now summarize all the distances seen in an escape accordion with additional half-wraps or accordion folds. We measure both the distance along the knot and the distance in the plane between vertices. 

\begin{lemma}\label{lem:dist} Assume we have a piece of ribbon of width $w=1$, and have folded a sequence of accordion folds with fold angle $\pi/3$ and angle $d$ as shown in Figure~\ref{fig:accord-details}. Let $v_S$ and $v_E$ be the start and end vertices of an escape accordion. Assume that additional half-wraps or accordion folds have been added after the escape accordion. Let $w_i$ be the vertex corresponding to the start of the $i$th half-wrap or accordion fold, as described in Remark~\ref{rmk:notation}. We find the following distances between the vertices of the corresponding knot diagram.

\begin{itemize}
\item $d_K(v_S,v_E) =\frac{1}{\cos(\nicefrac{\pi}{6})\sin(\nicefrac{\pi}{6})} = \frac{4}{\sqrt{3}}$, 
\item $d(v_S,v_E) = \frac{1}{\cos(\nicefrac{\pi}{6})}=\frac{2}{\sqrt{3}}$,
 \item $d_K(w_1, w_3) = 2d$,
\item $d(w_1, w_3) = 2d\sin(\nicefrac{\pi}{6}) = d$,
\item $d_K(w_1,w_{2n+1})=2nd$,  and
\item $d(w_1,w_{2n+1}) = 2nd\sin(\nicefrac{\pi}{6}) = nd$.
\end{itemize}

In addition, the distance between the odd fold line $o$ and even fold line $e$ is $\nicefrac{d\sqrt{3}}{2}$.
\end{lemma}

\begin{proof} Without loss of generality, we assume distance $d(v_E,P)=1$, then use trigonometry. A proof of these results for any fold angle $\theta$ can be found in \cite{Den-TP} (in Lemma~6 and Corollary~7). Simply substitute in $\pi/3$ for $\theta$.
\end{proof}

We now discuss how to construct $k$ half-twists. We follow the discussion in \cite{Den-TP}, but the half-twists are arranged a little differently here so we can more easily construct pretzel links. When constructing half-twists we will use the notation in Figure~\ref{fig:half-twists-new}, that is strand $AB$ winds around (straight) strand $CD$. 

\begin{const}[$k$ half-twists] \label{const:twists} In this construction, we start with two pieces of oriented ribbon of the same width labeled $AB$ and $CD$. We then construct $k$ half-twists between them as follows.

{\bf Case 1.}  Assume $k=0$. We begin by taking ribbons $AB$ and $CD$ and making accordion folds with fold angle $\pi/3$ and distance $d$ in each ribbon. 
\\ Step 1: We make an escape accordion in ribbon $AB$. This is shown on the left in Figure~\ref{fig:HTw-make}. Next, undo  the first fold of ribbon $AB$ at vertex $v_S$, this is shown in the center of Figure~\ref{fig:HTw-make}. We let vertex $X$ denote the final vertex of the folds that make up the escape accordion. This occurs  at distance $d$ before vertex $v_E$ on the knot diagram.

\begin{center}
    \begin{figure}[htpb]
    \begin{overpic}{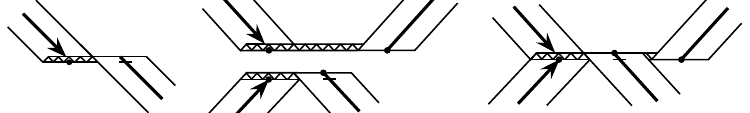}
    \put(1,13.5){$A$}
    \put(21,-0.5){$B$}
    \put(8.5,4.5){$v_S$}
    \put(15,8.5){\small{$v_E$}}
     \put(35,6.25){\small{$v_S$}}
     \put(42,6){\footnotesize{$X$}}
     \put(45,3.5){\footnotesize{$v_E$}}
    \put(29.5,-2){$A$}
    \put(48,-1.5){$B$}
    \put(27.5,15){$C$}
    \put(57,15.5){$D$}
    \put(50, 5.5){$Y$}
     \put(75,5.5){\small{$v_S$}}
     \put(80,5.5){\small{$v_E$}}
    \put(67,-1){$A$}
    \put(87.5,-0.5){$B$}
    \put(66,14){$C$}
    \put(97,14){$D$}
    \put(81,8.75){$X$}
     \put(90,4){$Y$}
    \end{overpic}
    \caption{Escape accordion $AB$ is placed over the long accordion fold $CD$ to create 0 half-twists.}
    \label{fig:HTw-make}
    \end{figure}
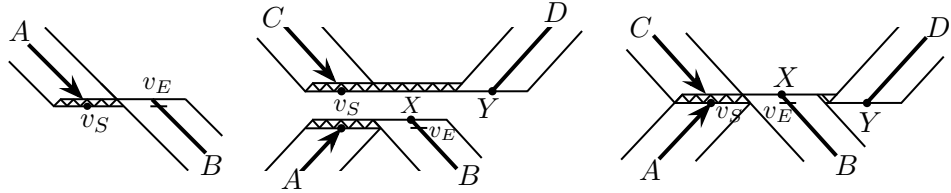
\end{center}

\noindent Step 2: We make an escape accordion in the second ribbon $CD$, and follow this with the same number of accordion folds as an escape accordion. We let $Y$ denote the final vertex of these folds and observe that $d_K(v_S, v_E) =d_K(v_E, Y)$. Fold end $D$ up at vertex $Y$ so that ends $C$ and $D$ lie on the same side of the accordion folds. 
\\ Step 3:  Next, move ribbon $AB$ over ribbon $CD$, so that the start  and end vertices ($v_S$ and $v_E$) of each escape accordion coincide and the corresponding odd and even fold lines coincide.  By construction, the end $D$ of ribbon $CD$ lies after the end $B$ of ribbon $AB$ as shown on the right in Figure~\ref{fig:HTw-make}.   There are 0 half-twists between ribbon $AB$ and ribbon $CD$.

\begin{center}
    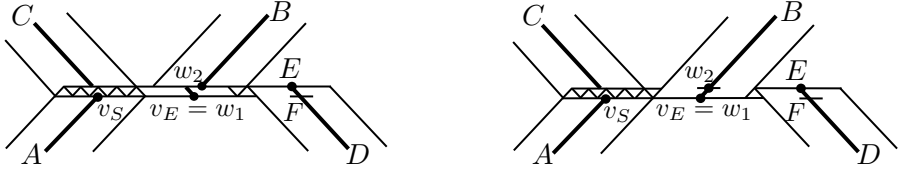
\begin{figure}[htpb]
    \begin{overpic}{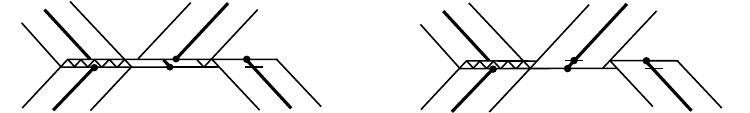}
     \put(12.5,4){\small{$v_S$}}
     \put(18,4){\small{$v_E=w_1$}}
     \put(20.5,8){\small{$w_2$}}
     \put(4.5,-1){$A$}
     \put(30.5,14){$B$}
     \put(3.5,13.5){$C$}
     \put(38.5,-1){$D$}
     \put(31.5, 8){$E$}
     \put(32,3.5){\small{$F$}}
     \put(65.5,3.75){\small{$v_S$}}
     \put(71,4){\small{$v_E=w_1$}}
     \put(74,8){\small{$w_2$}}
    \put(58,-1){$A$}
    \put(84,14){$B$}
    \put(57,13.5){$C$}
    \put(92,-1){$D$}
     \put(84.5, 7.75){$E$}
     \put(84.5,3.5){\small{$F$}}
    \end{overpic}
    \caption{On the left, end $B$ of escape accordion $AB$ is wrapped once under and behind the long accordion fold $CD$ to create a $+1$ half-twist. On the right, end $B$ folds once at the front of accordion fold $CD$ to create  a $-1$ half-twist.}
    \label{fig:HTw-1}
    \end{figure}
\end{center}

{\bf Case 2.} Assume $k=\pm1$.   To construct a $+1$ half-twist, start with 0 half-twists constructed as in Case 1. In ribbon $CD$, add one more accordion fold, and let $E$ denote the vertex at the final accordion fold in end $D$.  In ribbon $AB$,  add one left underfold at vertex $v_E=w_1$ so that end $B$ passes under the accordion folds of ribbon $CD$ as well. By construction, we see that $d_K(v_S,v_E)=d_K(w_2,E)$. This construction is illustrated on the left in Figure~\ref{fig:HTw-1}, and is precisely a $+1$ half-wrap of ribbon $AB$ about ribbon $CD$ as defined in Definition~\ref{defn:wrap}. To see the $+1$ half-twist between ribbons $AB$ and $CD$, simply move ribbon $AB$ down a bit.

To construct a $-1$ half-twist, start with a 0 half-twist as in Case 1.  In ribbon $CD$,  add one more accordion fold, and let $E$ denoted the vertex at the final accordion fold in end $D$. In ribbon $AB$, fold a left overfold at vertex $v_E=w_1$ so that end $B$ passes over the accordion folds of ribbons $AB$ and $CD$.  To see the $-1$ half-twist, simply move ribbon $AB$ down a bit.

\begin{center}
    \begin{figure}[htpb]
    \begin{overpic}{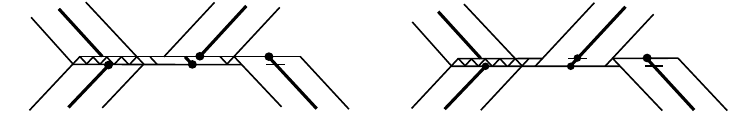}
     \put(14,4){\small{$v_S$}}
     \put(25,4){$w_3$}
     \put(24.25,9){\small{$w_4$}}
     \put(6.5,-1){$A$}
     \put(33,14){$B$}
     \put(5.5,13.5){$C$}
     \put(41,-1){$D$}
     \put(33.5, 8.5){$E$}
     \put(33.5,4){\small{$F$}}
     \put(64.5,3.75){\small{$v_S$}}
     \put(75,4){$w_3$}
     \put(74.5,8.5){\small{$w_4$}}
    \put(57,-1){$A$}
    \put(83.5,14){$B$}
    \put(56,13.5){$C$}
    \put(91.5,-1){$D$}
     \put(84, 8){$E$}
     \put(84,3.75){\small{$F$}}
    \end{overpic}
    \caption{Ribbon $AB$ is wrapped $2n+1$ times around the long accordion fold of ribbon $CD$. On the left we see a $+3$ half-twist and on the right a $-3$ half-twist.}
    \label{fig:HTw-odd}
    \end{figure}
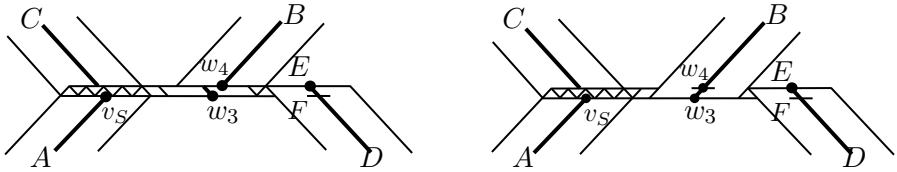
\end{center}

{\bf Case 3:} Assume $k=\pm(2n+1)$. To construct $+(2n+1)$ half-twists between ribbons $AB$ and $CD$, start with a $+1$ half-twist as constructed in Case 2. In ribbon $CD$ add $2n$ more accordion folds and let $E$ be the vertex corresponding to the final accordion fold. In ribbon $AB$ construct another $+2n$ half-wraps of end $B$ around the accordion folds of $AB$ and $CD$ (as in Definition~\ref{defn:wrap}). There are now $+(2n+1)$ half-wraps ending at vertex $w_{2n+1}$. As in Case 2, we have  $d_K(v_S,v_E)=d_K(w_{2n+2}, E)$. This construction is shown on the left in Figure~\ref{fig:HTw-odd} (but for $+3$ half-twists), and there are now $+(2n+1)$ half-twists between ribbons $AB$ and $CD$.  

On the right in Figure~\ref{fig:HTw-odd} we see the construction of $-(2n+1)$ half-twists  between ribbons $AB$ and $CD$.  The construction is very similar to the positive case. However, we start with a $-1$ half-twist as constructed in Case 2. After adding $2n$ more accordion folds to ribbon $CD$, we then construct another $-2n$ half-wraps of end $B$ around the accordion folds of $AB$ and $CD$. This completes the construction of $-(2n+1)$ half-twists between ribbons $AB$ and $CD$.

\begin{center}
    \begin{figure}[htpb]
    \begin{overpic}{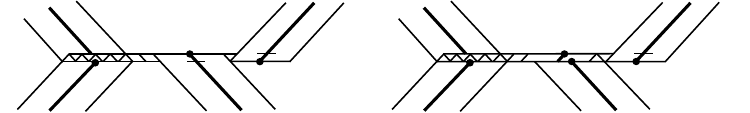}
     \put(13,5){\small{$v_S$}}
     \put(23.75,5){\small{$w_5$}}
     \put(23.5,9){{$w_4$}}
     \put(4,-1){$A$}
     \put(41,14){$D$}
     \put(4,13.5){$C$}
     \put(32,-1){$B$}
     \put(33.75, 8.5){\small{$H$}}
     \put(33.25,4.25){\small{$G$}}
     \put(63,5){\small{$v_S$}}
     \put(73.5,5){\small{$w_5$}}
     \put(74,9){\small{$w_4$}}
    \put(54,-1){$A$}
    \put(91,14){$D$}
    \put(54,13.5){$C$}
    \put(82,-1){$B$}
     \put(83.75, 8.5){\small{$H$}}
     \put(83.25,4.25){\small{$G$}}
    \end{overpic}
    \caption{Ribbon $AB$ is wrapped $2n$ times around the long accordion fold  of ribbon $CD$. On the left we see a $+4$ half-twist, and on the right,   a $-4$ half-twist.}
    \label{fig:HTw-even}
    \end{figure}
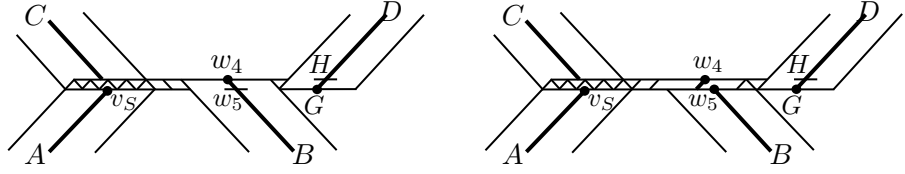
\end{center}

{\bf Case 4:} Assume $k=\pm 2n$. To construct $-2n$ half-twists between ribbons $AB$ and $CD$, start with a $0$ half-twist as constructed in Case 1. In ribbon $CD$ add $2n$ more accordion folds and let $G$ be the vertex corresponding to the final fold in end $D$. In ribbon $AB$ construct $+2n$ half-wraps of end $B$ around the accordion folds of $AB$ and $CD$ (as in Definition~\ref{defn:wrap}). By construction, we see that $d_K(v_S, v_E)=d_K(w_{2n+1}, G)$. This construction is shown on the left in Figure~\ref{fig:HTw-even} (but for $+4$ half-twists), and there are now $+2n$ half-twists between ribbons $AB$ and $CD$.  

On the right in Figure~\ref{fig:HTw-even} we see the construction of $-2n$ half-twists.  The construction is very similar to the positive case. We start with a $0$ half-twist as constructed in Case 1. After adding $2n$ more accordion folds to ribbon $CD$, we then construct $-2n$ half-wraps of end $B$ around the accordion folds of $AB$ and $CD$. This construction creates $-2n$ half-twists between ribbons $AB$ and $CD$.
\qed
\end{const}

\begin{remark}\label{rmk:vertices} Since this will be needed later, we pause emphasize which vertices lie on which lines. 
\begin{itemize}
\item Assume there are $\pm(2n+1)$ half-twists, between ribbons $AB$ and $CD$. Vertices $v_S$, $v_E$, $w_{2n+1}$ and $F$ lie on the odd fold line $o$, and vertices $w_{2n}$ and $E$ lie on the even fold line $e$. 
\item Assume there are $\pm 2n$ half-twists, between ribbons $AB$ and $CD$. Vertices $v_S$, $v_E$, $w_{2n+1}$, $Y$, $G$ lie on the odd fold line $o$, and vertices $X$, $w_{2n}$ $H$ lie on the even fold line $e$. 
\end{itemize}
\end{remark}

\begin{proposition} \label{prop:twists} Assume we have two pieces of ribbon labeled $AB$ and $CD$ with the same width and have folded a sequence of accordion folds with fold angle $\pi/3$ and distance $d$ in each ribbon. Assume for any $k\in\Z$, that we have folded $ k$ half-twists between ribbons $AB$ and $CD$ as in Construction~\ref{const:twists}. Then these $k$ half-twists can be constructed with folded ribbonlength at least $$\Rib(\text{$k$ half-twists})\geq \frac{12}{\sqrt{3}}+ (2|k|-1)d.$$
Equivalently, 
$\Rib(\text{$k$ half-twists})\geq \begin{cases} \frac{12}{\sqrt{3}}+ (4n-1)d & \text{ when $k=\pm 2n$}, \\
\frac{12}{\sqrt{3}}+ (4n+1)d & \text{ when $k=\pm(2n+1)$}.
\end{cases}$
\end{proposition}

Note that this proof has a similar flavor to the ones found in \cite{Den-TP}. However, the half-twists are slightly different, so we include the proof in full here.

\begin{proof} We compute the ribbonlength of the two pieces of ribbon labeled $AB$ and $CD$ used in Construction~\ref{const:twists}.  We measure each piece starting at vertex $v_S$ and ending at the final vertex of the accordion folds/wraps. We ignore any ribbon after these final vertices. We use Lemma~\ref{lem:dist} throughout this proof. 

 Assume $k=0$, and refer to Figure~\ref{fig:HTw-make}. Starting with ribbon $AB$, end $B$ follows the knot diagram from vertex $v_S$ in a zig-zag fashion through the escape accordion to vertex $X$, which we recall is one vertex before vertex $v_E$ in the knot diagram. The distance travels by end $B$ is $d_K(v_S, v_E) - d$.  Next, for ribbon $CD$, end $D$ follows the knot diagram from vertex $v_S$ to then travels in a zig-zag fashion through the escape accordion to vertex $v_E$, and then travels through another escape accordion to vertex $Y$. The distance traveled is $2d_K(v_S, v_E)$.  We now put these together 
 $$\Rib(0 \text{ half-twists}) \geq 3d_K(v_S, v_E) -d = \frac{12}{\sqrt{3}} -d. $$ 
 
 Assume $k=\pm 1$ and refer to Figure~\ref{fig:HTw-1}. To get a $\pm1$ half-twist, we start with the $0$ half-twist case and add one more accordion fold to ribbon $CD$ and one half-wrap  (appropriately oriented) to ribbon $AB$. Thus, starting with ribbon $AB$, end $B$ follows the knot diagram from vertex $v_S$ in a zig-zag fashion through the escape accordion to vertex $v_E=w_1$. The distance travels by end $B$ is $d_K(v_S, v_E)$.  Next, for ribbon $CD$, end $D$ follows the knot diagram from vertex $v_S$ to then travels in a zig-zag fashion through the accordion folds to vertex $v_E$, and then travels through another escape accordion plus one more fold to vertex $E$. The distance traveled is $2d_K(v_S, v_E)+d$.  We now put these together 
 $$\Rib(\pm 1 \text{ half-twists}) \geq 3d_K(v_S, v_E) +d = \frac{12}{\sqrt{3}} +d.$$

 We can proceed by induction. Assume the result is true for $k>0$ half-twists. We add one additional half-twist, by adding one more accordion fold to ribbon $CD$ and by adding one more half-wrap to ribbon $AB$. That is, we increase the ribbonlength by $+2d$. Thus the ribbonlength for $+(k+1)$ half-twists is 
 $$\Rib(k \text{ half-twists})\geq \frac{12}{\sqrt{3}}+ (2k-1)d +2d = \frac{12}{\sqrt{3}}+ (2(k+1)-1)d.$$ 
 A similar proof holds for $k<0$ half-twists, and thus the result holds for all integers.

\end{proof}

\begin{remark} In Proposition~\ref{prop:twists}, we have chosen to measure the folded ribbonlength of the $k$ half-twists from the first fold at vertex $v_S$ to the final fold in each ribbon. We made this choice to simplify later calculations. When computing the folded ribbonlength of a link containing a strand of $k$ half-twists, we can easily measure the folded ribbonlength of the additional ribbon used by starting at these places.
\end{remark}
We can now proof Theorem~\ref{thm:twists}: Assume $k$ is a nonzero integer, and assume that we have folded $k$ half twists between two pieces of ribbon of the same width. Then, ignoring the ends of the ribbons, these $k$ half-twists can be constructed with folded ribbonlength at least
$$\Rib(k \text{ half-twists})= \frac{12}{\sqrt{3}}+ \epsilon, \text{  for all $\epsilon >0$}. $$

\begin{proof}[Proof of Theorem~\ref{thm:twists}] The inequality in Proposition~\ref{prop:twists} comes from considering the additional ribbonlength of the end of each ribbon occurring before the start vertex $v_S$ and after the vertex corresponding to the final fold. If we ignore the ribbonlength of these ends, then once $\epsilon>0$ is chosen, simply set $d$ small enough in Proposition~\ref{prop:twists}. 
\end{proof}

Before we move on to the construction of a folded ribbon pretzel link, we need two more lemmas that will simplify our ribbonlength computations. Part of first lemma initially appeared in \cite{Den-TP} (as part of the proof of Theorem 11). We have added a second case and included the proof here for ease of use.

\begin{center}
    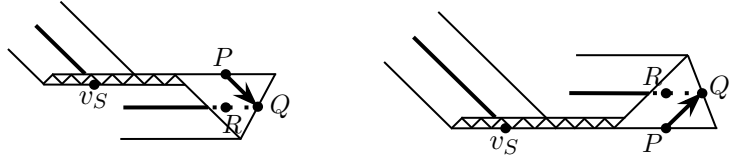
\begin{figure}[htpb]
    \begin{overpic}{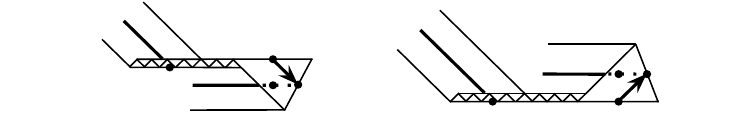}
    \put(21,4){$v_S$}
    \put(35,8){\small{$P$}}
    \put(41,3){\small{$Q$}}
    \put(36,1){\small{$R$}}
    \put(64,-0.75){$v_S$}
    \put(80,-1){\small{$P$}}
    \put(87,5){\small{$Q$}}
    \put(80,6){\small{$R$}}
    \end{overpic}
    \caption{The end of the ribbon is folded back towards the start vertex $v_S$ so that one side-edge of the ribbon  coincides with the fold line containing vertex $P$. }
    \label{fig:accord-end-fold}
    \end{figure}
\end{center}

\begin{lemma}\label{lem:triangle} Take a piece of ribbon of width $w=1$. Fold an escape accordion folded with fold angle $\pi/3$ and distance $d$, then fold any number ($\geq 0$) of  additional half-wraps or accordion folds. Let $P$ be the vertex in the knot diagram corresponding to the final fold of the half-wraps/accordion folds.  Fold the end of the ribbon at point $Q$ back towards the start vertex $v_S$ with fold angle $\pi/3$ so that one side-edge of the ribbon coincides with the fold line containing point $P$. We let $R$ be the point on the final segment of the knot diagram such that $PR$ is perpendicular to $PQ$.  This construction is illustrated in Figure~\ref{fig:accord-end-fold}.
The distances in the figure are as follows:
\begin{compactitem}
\item $d(P,R)=\frac{1}{2}$
\item $d(P,Q)= \frac{1}{\sqrt{3}}$
\item $d(Q,R) = \frac{1}{2\sqrt{3}}$
\end{compactitem}

\end{lemma}

\begin{proof}
 In both cases, we have a right triangle $\triangle PRQ$, with angle $\angle PQR=\pi/3$. Since the width $w=1$, we have $d(P,R)=\frac{1}{2}$. The hypotenuse has length $d(P,Q) = \frac{1}{2\sin(\pi/3)} = \frac{1}{\sqrt{3}}$, and the final edge has length $d(R,Q)=\frac{1}{2}\tan(\pi/6) = \frac{1}{2\sqrt{3}}$.  We note that while Figure~\ref{fig:accord-end-fold} shows an underfold at vertex $Q$, the same proof works if there is an overfold at $Q$. 
\end{proof}

When constructing folded ribbon knots, we often encounter two overlapping parallel pieces of ribbon which we need to join together. The simplest way to do this is with a fold of fold angle~0. Now fold lines with fold angle~0 are perpendicular to the knot diagram, and  must occur at a minimum distance from the vertex of an existing fold line. Why? By Definition~\ref{def:FR}, fold lines can not intersect in a folded ribbon knot. The next lemma gives this distance.

\begin{center}
    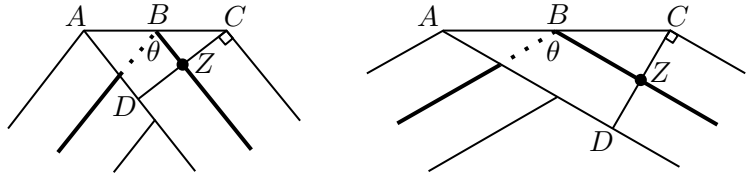
\begin{figure}[htpb]
    \begin{overpic}{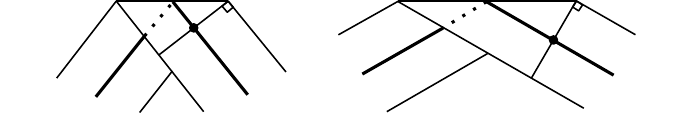}
    \put(15,17){$A$}
    \put(24.5, 17.25){$B$}
    \put(24.5, 13){$\theta$}
    \put(33.5,17){$C$}
    \put(20.5,6.5){$D$}
    \put(30,11.5){$Z$}
    \put(56,17){$A$}
    \put(71, 17.25){$B$}
    \put(71, 13){$\theta$}
    \put(85,17){$C$}
    \put(76,2.5){$D$}
    \put(83,10.5){$Z$}
    \end{overpic}
    \caption{Line $CD$ shows the closest location of a fold with fold angle 0 to an existing fold line $AC$.  }
    \label{fig:acute-obtuse}
    \end{figure}
\end{center}

\begin{lemma}\label{lem:foldline0} Following the notation in Figure~\ref{fig:acute-obtuse}, take a piece of ribbon of width $w=1$ with a fold line $AC$ formed from a fold angle $\theta$ at vertex $B$ of the knot diagram.  Assume that line segment $CD$ is perpendicular to the knot diagram and intersects it in point $Z$. Then distance $d(B,Z) = \frac{1}{2}\tan(\frac{\theta}{2})$. In particular, if $\theta=\frac{\pi}{3}$, then $d(B,Z)=\frac{1}{2\sqrt{3}}$.
\end{lemma}

\begin{proof} Using the geometry of the fold we see right triangle $\triangle BZC$  triangle has angle $\angle ZBC = \frac{\pi}{2}-\frac{\theta}{2}$ and angle $BCZ=\frac{\theta}{2}$. Since the ribbon has width $w=1$, then $d(C,Z)=\frac{1}{2}$, and distance $d(B,Z) = \frac{1}{2}\tan(\frac{\theta}{2})$.
\end{proof}


\section{Ribbonlength of pretzel links}\label{sect:ribbon}

In this section we give a construction for a folded ribbon pretzel link $P(p,q,r)$, we then prove that this construction has bounded folded ribbonlength (that is, does not depend on $p$, $q$ or $r$).  The construction has two features. First, we fold three strands of $p$, $q$ and $r$ half-twists using Construction~\ref{const:twists}. Second, we join these strands of half-twists together in a way that minimizes the amount of ribbon used to join them. The key idea here is to arrange the three strands of half-twists so they lie on top of one another. We do this concertina style as shown in Figure~\ref{fig:pretzel-fold}. Starting from the left image in Figure~\ref{fig:pretzel-fold}, we fold the $q$ half-twists over to the left so they lie on the $p$ half-twists. We then fold the $r$ half-twists over to the right so they lie over the $q$ half-twists, as shown in the center and right of Figure~\ref{fig:pretzel-fold}.

\begin{center}
    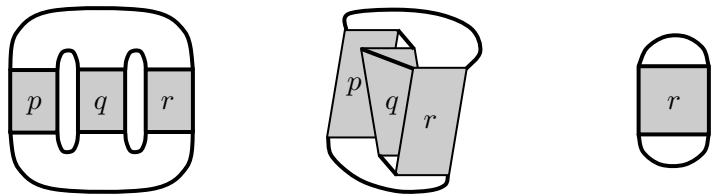
\begin{figure}[htpb]
    \begin{overpic}{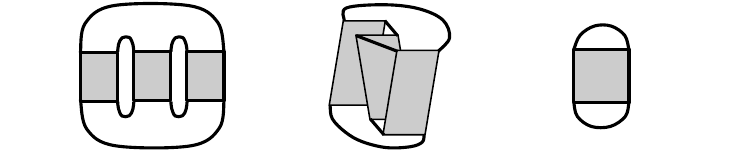}
    \put(12.5,9){$p$}
    \put(19.5,9){$q$}
    \put(26.5,9){$r$}
    \put(46,11){$p$}
    \put(50,9){$q$}
    \put(54,7){$r$}
    \put(79.5,9){$r$}
    \end{overpic}
    \caption{The three strands of half-twists in the $P(p,q,r)$ pretzel link are arranged in a new way. Figure reused with permission from \cite{Den-FRF}.}
    \label{fig:pretzel-fold}
    \end{figure}
\end{center}

Before we begin, we observe that for a $P(p,q,r)$ pretzel link we can reduce the number of possibilities of integers $p$, $q$, $r$ that we need to think about. The length of $k$ half-twists does not change if $k$ is positive or negative, so we can ignore that choice.  As shown in Construction~\ref{const:twists}, the arrangement of the ends of the strand of $k$ half-twists changes depending on whether $k$ is odd or even. We can thus reduce our constructions down to two cases. Case 1: $p$, $q$ and $r$  have the same parity found in Construction~\ref{const:pretzel-odd}. Case 2: one of $p$, $q$, and $r$ has the opposite parity to the other two found in Construction~\ref{const:pretzel-parity}.

We use Construction~\ref{const:twists} to construct the half-twists. As a reminder, this construction uses two equal width pieces of ribbon labeled $AB$ and $CD$. Ribbon $CD$ is folded in a long accordion fold, and ribbon $AB$ is wound around this using a series of half-wraps. We will keep using this notation for the ends of the half-twists, and add a subscript to indicate which half-twist we are referring to. Thus the ends of the $p$ half-twists are labeled $A_p$, $B_p$, $C_p$, $D_p$. 

\begin{remark}\label{rmk:length}  Suppose that integers $p$ and $r$ have the same parity and that $|p|<|r|$.  Suppose we have constructed both $p$ and $r$ half-twists following Construction~\ref{const:twists}. Observe that we can add an even number ($|r|-|p|$) of accordion folds to the ends $B_p$ and $D_p$ so that the $p$ and $r$ half-twists now have the same folded ribbonlength. In particular, vertex $w_r$ is the vertex corresponding to the final fold in both ends $B_p$ and $B_r$.   Thus, this move adds folded ribbonlength to the ends $B_p$ and $D_p$,  but does not change the number of half-twists. We will use this technique later, since this will simplify many constructions without adding too much extra folded ribbonlength. (The amount added is a multiple of the distance $d$ --- which can be made arbitrarily small.)
\end{remark}

\begin{const}\label{const:pretzel-odd} [$P(p,q,r)$ pretzel links --- same parity] We construct a folded ribbon $P(p,q,r)$ pretzel link, where $p$, $q$ and $r$ all have the same parity. 

Assume that $p$, $q$, and $r$ are all odd.  From Remark~\ref{rmk:pretzel} we can assume, without loss of generality, that $|p|\leq |q| \leq |r|$.   Use Construction~\ref{const:twists} to construct three strands of $p$, $q$ and $r$ half-twists out of pieces of ribbon with equal width. Following Remark~\ref{rmk:length}, add additional accordion folds to the ends $B_p$, $D_p$, $B_q$, $D_q$ of the $p$ and $q$ half-twists so that they have the same folded ribbonlength as the $r$ half-twist while maintaining the same number of half-twists.  As mentioned above, we do this to simplify our ribbonlength calculations without adding too much additional folded ribbonlength.

\begin{center}
    \begin{figure}[htpb]
    \begin{overpic}{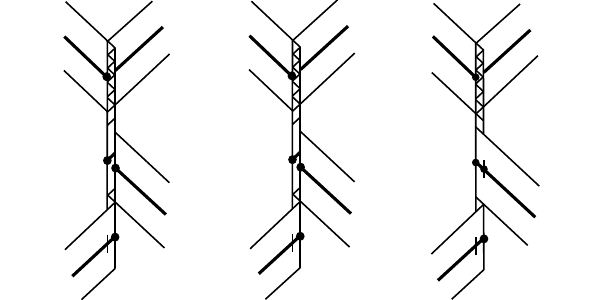}
    \put(6, 45){$A_p$}
    \put(28,45){$C_p$}
    \put(28,13){$B_p$}
    \put(8,2){$D_p$}
    \put(14, 35){\small{$v_S$}}
    \put(13,23){\small{$w_r$}}
    \put(20,22){\footnotesize{$w_{r+1}$}}
    \put(20,9){$E_p$}
   \put(37, 45){$A_q$}
    \put(58,46){$C_q$}
    \put(58,13){$B_q$}
    \put(39,2){$D_q$}
    \put(45, 35){\small{$v_S$}}
    \put(44,23){\small{$w_r$}}
    \put(51,22){\footnotesize{$w_{r+1}$}}
    \put(51,9){$E_q$}
    \put(68, 45){$A_r$}
    \put(89,45){$C_r$}
    \put(89,13){$B_r$}
    \put(69,2){$D_r$}
    \put(75, 35){\small{$v_S$}}
    \put(74.5,23){\small{$w_r$}}
    \put(81.5,22){\footnotesize{$w_{r+1}$}}
    \put(81,9){$E_r$}
    \end{overpic}
    \caption{The $p$, $q$ and $r$ half-twists ready to be joined to form a folded ribbon $P(p,q,r)$ pretzel link.}
    \label{fig:pretzel-o-pqr}
    \end{figure}
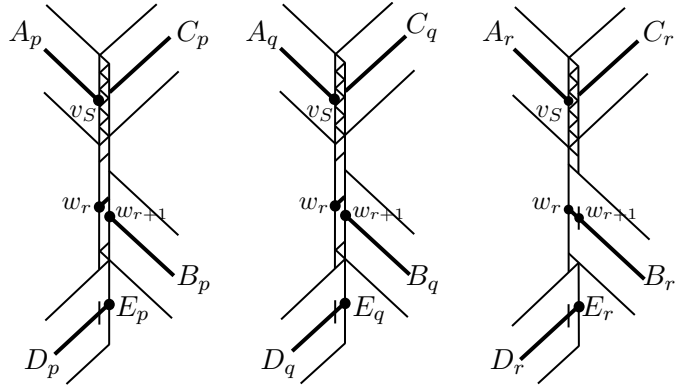
\end{center}

Figure~\ref{fig:pretzel-o-pqr} shows the three half-twists at this stage of the construction. Looking at the $r$ half-twists, we note that vertex $w_r$ is the vertex corresponding to the final accordion fold of end $B_r$, and vertex $E_r$ is the vertex corresponding to the final accordion fold of end $D_r$. We use similar notation for the $q$ and $r$ half-twists, where $w_r$ is the final vertex in each of the half-twists due to the extra accordion folds.  Figure~\ref{fig:pretzel-o-pqr} also shows strands of half-twists that are positive and negative.

In order to complete the pretzel knot, we join the ends of the strands of half-twists in a cyclic order (as shown in Figure~\ref{fig:pretzel-fold}). Namely, we need to join:
\begin{compactenum}
\item ends $C_p$ and $A_q$, and ends $B_p$ and $D_q$ of the $p$ and $q$ half-twists, 
\item ends $C_q$ and $A_r$, and ends $B_q$ and $D_r$ of the $q$ and $r$ half-twists, 
\item ends $A_p$ and $C_r$, and ends $D_p$ and $B_r$ of the $p$ and $r$ half-twists.
\end{compactenum}

{\bf Join 1:} Joining ends $C_p$ and $A_p$, and ends $B_p$ and $D_q$ of the $p$ and $q$ half-twists. 
\\ 
Step 1: Rotate the strand of $q$ half-twists by $\pi$ radians along the axis of the accordion folds (so they are now upside down). Next, move them so they lie to the left of the $p$ half-twists. This is shown on the left in Figure~\ref{fig:pretzel-o-qp}. Notice in the figure that the odd fold line $o_q$ of the $q$ half-twists is to the right of the even fold line $e_q$, and the odd fold line $o_p$ of the $p$ half-twists is to the left of the even fold line $e_p$. 
\\ 
Step 2:  Place the $q$ half-twists on top of the $p$ half-twists so that the odd fold lines coincide ($o_q=o_p=o$), and so that the start vertices $v_S$ coincide. This is shown on the right in Figure~\ref{fig:pretzel-o-qp}. There is a lot going on in this figure, notice that the vertices $w_r$ of the $q$ and $p$ half-twists also coincide. The ends we need to join together ($C_p$ to  $A_q$,  and $B_p$ to $D_q$) all lie on the right side of the overlapping half-twists.

\begin{center}
    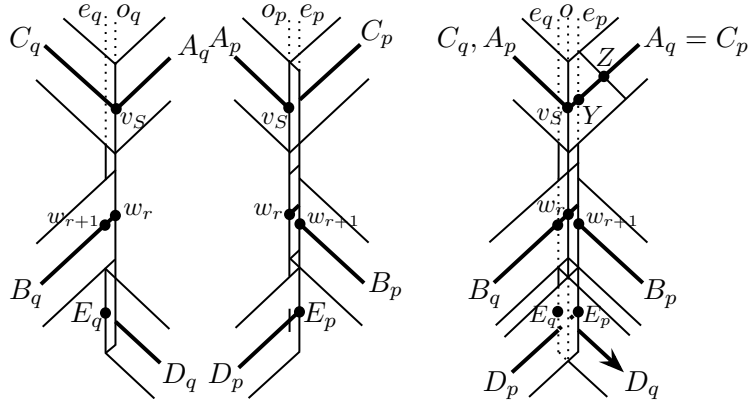
\begin{figure}[htpb]
    \begin{overpic}{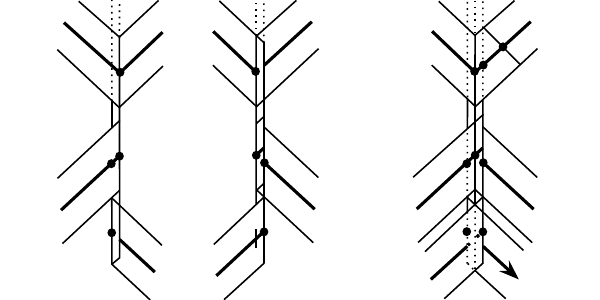}
    \put(15,50){$e_q$}
    \put(20,50){$o_q$}
    \put(6, 46){$C_q$}
    \put(28,45){$A_q$}
    \put(6,13){$B_q$}
    \put(26,2){$D_q$}
    \put(20.5, 36){\small{$v_S$}}
    \put(21,24){\small{$w_r$}}
    \put(11,23){\footnotesize{$w_{r+1}$}}
    \put(14,10){$E_q$}
   
    \put(44,50){$e_p$}
    \put(39,50){$o_p$}
   \put(32, 46){$A_p$}
    \put(52,47){$C_p$}
    \put(53,14){$B_p$}
    \put(32,2){$D_p$}
    \put(39, 36){\small{$v_S$}}
    \put(38,24){\small{$w_r$}}
    \put(45,23){\footnotesize{$w_{r+1}$}}
    \put(44.5,10){$E_p$}
    \put(74,49.5){$e_q$}
    \put(78,50){$o$}
    \put(81,49.5){$e_p$}
    \put(89, 46){$A_q=C_p$}
    \put(62,46){$C_q,A_p$}
    \put(89,13){$B_p$}
    \put(66,13){$B_q$}
    \put(68,1){$D_p$}
    \put(86.5,1){$D_q$}
    \put(75, 36.5){\small{$v_S$}}
    \put(81,36){\small{$Y$}}
    \put(82.75, 43.5){\small{$Z$}}
    \put(75,24.5){\small{$w_r$}}
    \put(81.5,23){\footnotesize{$w_{r+1}$}}
    \put(74,10){\footnotesize{$E_q$}}
    \put(81,10){\footnotesize{$E_p$}}
    \end{overpic}
    \caption{On the left, the flipped $q$ half-twists have been placed to the left of the $p$ half-twists.  On the right, the $q$ half-twists have been carefully positioned on top of the $p$ half-twists.}
    \label{fig:pretzel-o-qp}
    \end{figure}
\end{center}
\noindent Step 3: Ends $A_q$ to $C_p$ coincide. So we join them together along a fold line with fold angle 0, one end of which touches the even fold line $e_p$.  This is shown in Figure~\ref{fig:pretzel-o-qp} right, with the fold line passing through point $Z$ on the knot diagram. Point $Y$ is the point near $v_S$ where the knot diagram intersects with even fold line $e_p$, thus $d(v_S,Y)=d$.  From Lemma~\ref{lem:foldline0} we have $d(Y,Z)=\frac{1}{2\sqrt{3}}$. Thus this join adds $2d(v_S,Z)=2d + \frac{1}{\sqrt{3}}$ units to the folded ribbonlength. 
\\ 
Step 4: We join end $D_q$ to $B_p$, by starting with end $D_q$ at vertex $E_q$, then fold end $D_q$ with a left underfold with fold angle $\pi/3$ at vertex $T$, as shown in Figure~\ref{fig:pretzel-ejoin1}. We have chosen vertex $T$ so the lower side-edge of the ribbon coincides with the even fold line $e_q$. Point $S$ has been chosen so that line segment $SE_q$ is perpendicular to the even fold line $e_q$. We have fulfilled the hypotheses of Lemma~\ref{lem:triangle}, thus distance $d(E_q,T) = \frac{1}{\sqrt{3}}$,

\begin{center}
    \begin{figure}[htpb]
    \begin{overpic}{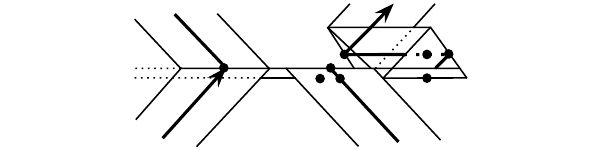}
    \put(24,23){$A_q$}
    \put(23,0){$C_q$}
    \put(67,0){$B_q$}
    \put(66,24){$D_q$}
    \put(70,8){\small{$E_q$}}
    \put(18,11){$e_q$}
    \put(18,14){$o_q$}
    \put(38, 15){\small{$v_S$}}
    \put(50.5,15){\small{$w_r$}}
    \put(51,9.5){\footnotesize{$w_{r-1}$}}
    \put(58,11){\footnotesize{$w_{r+1}$}}
    \put(54,15.25){\small{$R$}}
    \put(70,17){\small{$S$}}
    \put(76,16){\small{$T$}}
   \end{overpic}
    \caption{The flipped $q$ half-twists have been rotated to show how end $D_p$ is folded to join end $B_p$. The $p$ half-twists have been removed for clarity.}
    \label{fig:pretzel-ejoin1}
    \end{figure}
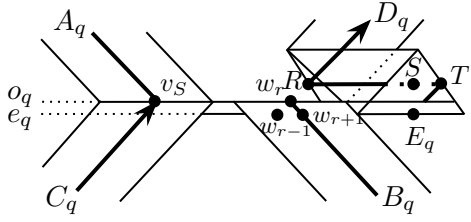
\end{center}
\noindent 
Step 5: We make a second fold in end $D_q$, this time a right underfold with fold angle $\pi/3$ at vertex $R$. This vertex has been chosen so the end $D_q$ coincides with end $B_p$. This means points $w_{r-1}$, $w_r$ and $R$ are collinear.  We join ends $B_p$ to end $D_q$ together after vertex $R$ with a fold of fold angle 0, such that one end of this fold line coincides with the fold line through $R$. This is shown in Figure~\ref{fig:pretzel-ejoin1}, albeit not to scale. The geometry of this ``join" fold line is the same as in Lemma~\ref{lem:foldline0}, thus the distance from $R$ to the fold line is $\frac{1}{2\sqrt{3}}$.
\\ 
Step 6: We show that the ribbonlength needed to join ends $B_p$ and $D_q$ is $\frac{5}{\sqrt{3}}$. Start with end $D_q$, and observe from the geometry of the construction that there is a parallelogram $w_{r-1}E_qTR$. We know from Lemma~\ref{lem:dist} that $d(w_{r+1},E_q)=d(v_S,v_E)=\frac{2}{\sqrt{3}}$, and $d(w_{r-1},w_{r+1})=d$. Thus, $d(T,R)=\frac{2}{\sqrt{3}}+d$. Thus up to vertex $R$, end $D_q$ adds $d(E_q,T) + d(T,R)=\frac{3}{\sqrt{3}}+d$ extra ribbonlength.   We see end $B_p$ adds on $d(w_r,R) = \frac{1}{\sqrt{3}} -d$ to the folded ribbonlength, since $d(w_{r-1}, R)=d(E_q,T)=\frac{1}{\sqrt{3}}$. After vertex $R$, use Step~5 to deduce that both ends combined add $\frac{1}{\sqrt{3}}$ to the folded ribbonlength.
\\
Step 7: Altogether we see, from Steps 3--6, that Join 1 adds $2d + \frac{6}{\sqrt{3}}$ to the folded ribbonlength.

Note: In Step 5, if we fold end $D_q$ at vertex $R$ with a left underfold towards vertex $w_r$, the fold angle is $2\pi/3$. One end of the fold line lies on the line $e_q$. However the end $B_p$ starts at vertex $w_r$ on the line $o=o_q$. Thus we can not join ends $D_q$ and $B_p$ without some additional folding.  The choices we have made in Steps 4 and 5 eliminate this concern.

{\bf Join 2:} Joining ends $C_q$ and $A_r$, and ends $B_q$ and $D_r$ of the $q$ and $r$ half-twists.
\\ 
Step 1: We now join the flipped $q$ half-twists to the $r$ half-twists. These are shown on the left of Figure~\ref{fig:pretzel-o-qr} (over the page).  We next place the $r$ half-twists on top of the flipped $q$ half-twists so that the odd fold lines coincide ($o_q=o_r=o$), and so that the start vertices $v_S$ coincide. This means the vertices $w_r$ of the $q$ and $r$ half-twists also coincide. This is shown on the right in Figure~\ref{fig:pretzel-o-qr}.  This time, the ends we need to join together ($C_q$ to $A_r$ and $B_q$ to $D_r$) all  lie on the left side of the overlapping twists.
\\
Step 2: Ends $C_q$ to $A_r$ coincide, so we join them with a fold line of fold angle $0$, one end of which starts at the even fold line $e_q$.  This is shown in Figure~\ref{fig:pretzel-o-qr} right, with the fold line passing through point $Z$ on the knot diagram. Following similar reasoning to Step 3 in Join 1,  joining $C_q$ to $A_r$ adds $2d(v_S,Z)=2d + \frac{1}{\sqrt{3}}$ units to the folded ribbonlength.

\begin{center}
    \begin{figure}[htpb]
    \begin{overpic}{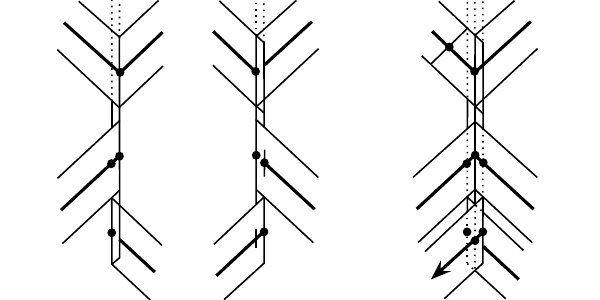}
    \put(15,50){$e_q$}
    \put(20,50){$o_q$}
    \put(6, 46){$C_q$}
    \put(28,45){$A_q$}
    \put(6,13){$B_q$}
    \put(26,2){$D_q$}
    \put(20.5, 36){\small{$v_S$}}
    \put(21,24){\small{$w_r$}}
    \put(11,23){\footnotesize{$w_{r+1}$}}
    \put(14,10){$E_q$} 
    \put(44,50){$e_r$}
    \put(39,50){$o_r$}
   \put(32, 46){$A_r$}
    \put(52,47){$C_r$}
    \put(53,14){$B_r$}
    \put(32,2){$D_r$}
    \put(39, 36){\small{$v_S$}}
    \put(38,24){\small{$w_r$}}
    \put(45,23){\footnotesize{$w_{r+1}$}}
    \put(44.5,10){$E_r$}
    \put(74,49.5){$e_q$}
    \put(78,50){$o$}
    \put(81,49.5){$e_r$}
    \put(89, 46){$C_r$}
    \put(61,46){$C_q=A_r$}
    \put(89,13){$B_r$}
    \put(66,13){$B_q$}
    \put(68,1){$D_r$}
    \put(86.5,1){$D_q$}
    \put(75, 36.5){\small{$v_S$}}
    \put(73.25, 38.75){\footnotesize{$Z$}}
    \put(79.5,25){\small{$w_r$}}
    \put(70,23.5){\footnotesize{$w_{r+1}$}}
    \put(74,10){\footnotesize{$E_q$}}
    \put(81,10){\footnotesize{$E_r$}}
    \end{overpic}
    \caption{On the left, the flipped $q$ half-twists are to the left of the $r$ half-twists.  On the right, the $r$ half-twists have been carefully placed on top of the $q$ half-twists.}
    \label{fig:pretzel-o-qr}
    \end{figure}
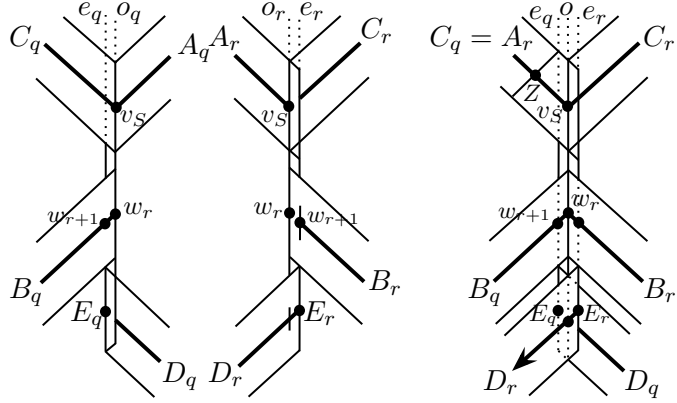
\end{center}
\noindent 
Step 3: We join end $D_r$ to $B_q$ using a similar fold pattern to those described in Steps 4 and 5 of the Join 1 case. However, this time end $D_r$ has a right underfold with fold angle $\pi/3$ followed by a left underfold with fold angle $\pi/3$. The same computations show that the amount of ribbonlength added here is $\frac{5}{\sqrt{3}}$. 
\\
Step 4: Altogether, we see Join 2 adds $2d + \frac{6}{\sqrt{3}}$ to the folded ribbonlength.

{\bf Join 3:} Joining ends $A_p$ and $C_r$, and ends $D_p$ and $B_r$ of the $p$ and $r$ half-twists.
\\ 
Step 1: At this point in the construction, the $r$ half-twists lie over the $p$ half-twists (with the $q$ half-twists between them). Here, the start vertices $v_S$ and the vertices $w_r$ coincide, and the respective odd and even fold lines coincide, so that $o_p=o_r=o$ and $e_p=e_r=e$. This is shown on the left in  Figure~\ref{fig:pretzel-o-pr}.
\\
Step 2: We join at ends $A_p$ and $C_r$. First, fold end $A_p$ at vertex $v_S$ along a fold line which is parallel to the odd fold line $o$, so that end $A_p$ lies under everything. At this stage, end $A_p$ coincides with end $C_r$. Second, join ends $A_p$ and $C_r$ with  a fold line of fold angle $0$, one end of which intersects the even fold line $e$. As before this new fold line passes through point $Z$ on the knot diagram.  Using the same reasoning as in Joins 1 and 2, joining $A_p$ and $C_r$ adds $2d+\frac{1}{\sqrt{3}}$ to the folded ribbonlength.

\begin{center}
    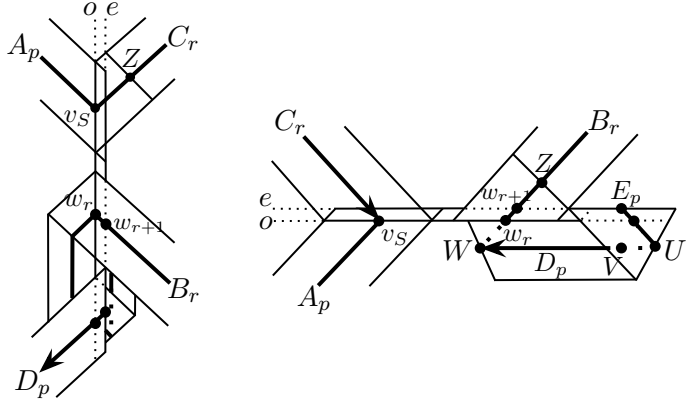
\begin{figure}[htpb]
    \begin{overpic}{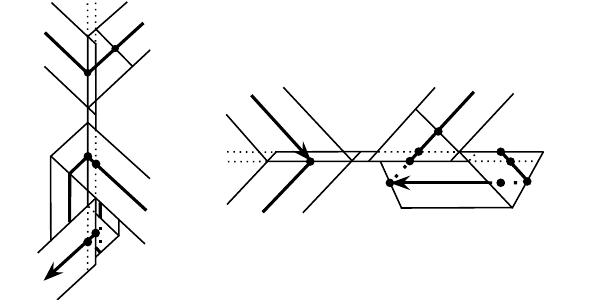}
    \put(16,50){$e$}
    \put(13,50){$o$}
    \put(3, 45){$A_p$}
    \put(24,46){$C_r$}
    \put(4,1){$D_p$}
    \put(24,13){$B_r$}
   \put(18,43.25){\small{$Z$}}
    \put(10.5, 36){\small{$v_S$}}
    \put(10.5,25){\small{$w_r$}}
    \put(17,22){\footnotesize{$w_{r+1}$}}
    \put(36,25){$e$}
    \put(36,22){$o$}
    \put(41, 12){$A_p$}
    \put(38,35){$C_r$}
    \put(72,16.5){\small{$D_p$}}
    \put(79,35){$B_r$}
    \put(52, 20.5){\small{$v_S$}}
    \put(68,20.5){\small{$w_r$}}
    \put(65,26){\footnotesize{$w_{r+1}$}}
    \put(72, 29.5){\small{$Z$}}
    \put(82,26){\small{$E_p$}}
    \put(89,18){$U$}
    \put(81,16){\small{$V$}}
    \put(60,18){$W$}
   \end{overpic}
    \caption{On the left, the $r$ half-twists lie directly over the $p$ half-twists (with the $q$ half-twists between them). On the right, the $r$ half-twists have been rotated to show how end $D_p$ is folded to join end $B_r$.}
    \label{fig:pretzel-o-pr}
    \end{figure}
\end{center}
\noindent 
Step 3: We next join end $D_p$ to $B_r$. Start with end $D_p$ at vertex $E_p$, then fold a right underfold with fold angle $\pi/3$ in end $D_p$ at vertex $U$. Here,  $U$ is chosen so that the upper side-edge of ribbon $D_p$ coincides with the even fold lines $e$.  We add another right underfold with fold angle $\pi/3$ at vertex $W$. Here $W$ has been chosen so that $W$, $w_r$ and $w_{r+1}$ are collinear. After vertex $W$, end $D_p$ coincides with end $B_r$. We join end $D_p$ with end $B_r$ with a fold line with fold angle 0,  one end of which touches the even fold line $e$. This new fold line intersects the knot diagram at point $Z$. The geometry of this ``join'' fold line is the same as in Lemma~\ref{lem:foldline0}, thus $d(w_{r+1},Z)=\frac{1}{2\sqrt{3}}$.
\\
Step 4: We show that the ribbonlength needed to join ends $D_p$ and $B_r$ is $\frac{6}{\sqrt{3}} +d$.  To find the folded ribbonlength added by end $D_p$ we traverse three edges along the trapezoid $E_pUWw_{r+1}$.   We know from Construction~\ref{const:twists} and Lemma~\ref{lem:dist} that $d(w_{r+1},E_p)=d(v_S,v_E)=\frac{2}{\sqrt{3}}$. Suppose that we construct point $V$ on line segment $WU$ so that $E_pV$ is perpendicular to it. Then in right  triangle $\triangle E_pVU$, we know from Lemma~\ref{lem:triangle} that $d(E_p,U)=\frac{1}{\sqrt{3}}$ and $d(U,V)=\frac{1}{2\sqrt{3}}$. Using symmetry we see end $D_p$ traverses the knot from $E_p$ to $w_{r+1}$ in distance $d_K(E_p,w_{r+1})= \frac{2}{\sqrt{3}} + \frac{2}{2\sqrt{3}} + \frac{2}{\sqrt{3}} =\frac{5}{\sqrt{3}}$. We also know $d(w_{r+1},Z)=\frac{1}{2\sqrt{3}}$. Therefore, end $D_p$ adds $\frac{5}{\sqrt{3}} + \frac{1}{2\sqrt{3}}$, while end $B_r$ adds $d(w_r,w_{r+1})+d(w_{r+1},Z)=d+\frac{1}{2\sqrt{3}}$ to the folded ribbonlength.
\\
Step 5: Altogether we see, from Steps 2--4, that Join 3 adds $3d+\frac{7}{\sqrt{3}}$ to the folded ribbonlength. 

We end this construction by observing that the case where $p$, $q$ and $r$ are all even is almost identical, however ends $A$ and $B$ are one side of the half-twists and ends $C$ and $D$ are on the other.  
\qed
\end{const}

\begin{const}\label{const:pretzel-parity} [$P(p,q,r)$ pretzel links --- one of opposite parity] We construct a $P(p,q,r)$ pretzel link, where one of $p$, $q$, $r$ is of opposite parity to the other two. From Remark~\ref{rmk:pretzel}, we can assume without loss of generality, that $p$ and $r$ are odd and $q$ is even. Use Construction~\ref{const:twists} to construct strands of $p$, $q$ and $r$ half-twists.  Let us assume that $|p|\leq |r|$ and $|q|\leq |r|$. Following Remark~\ref{rmk:length}, add additional accordion folds to the ends $B_p$, $D_p$ of the $p$ half-twists so the $p$ and $r$ half-twists have the same ribbonlength. Also add extra accordion folds as needed to the ends $B_q$, $D_q$ of the $q$ half-twists so that the $q$ half-twists are $2d$ units of ribbonlength shorter than the $r$ half-twists. This means that if the vertex of the final fold of the $r$-half-twists is $w_r$, then the vertex of the final fold of the $q$ half-twists is $w_{r-1}$. This is illustrated in Figure~\ref{fig:pretzel-pqr}. We make these assumptions to simplify our ribbonlength calculations without adding too much additional ribbonlength. Moreover, we can make these assumptions without loss of generality. For example, if $|p|\leq|r| \leq |q|$, then we can add additional accordion folds to the $p$ and $r$ half-twists until the geometry matches the assumptions above.

\begin{center}
    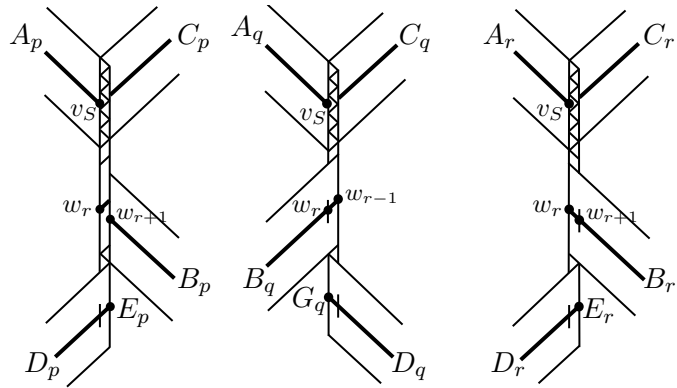
\begin{figure}[htpb]
    \begin{overpic}{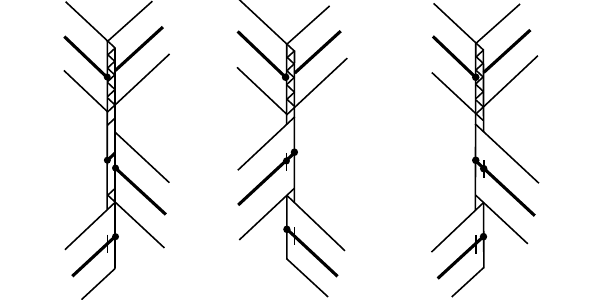}
    \put(6, 45){$A_p$}
    \put(28,45){$C_p$}
    \put(28,13){$B_p$}
    \put(8,2){$D_p$}
    \put(14, 35.5){\small{$v_S$}}
    \put(13,23){\small{$w_r$}}
    \put(20,22){\footnotesize{$w_{r+1}$}}
    \put(20,9){$E_p$}
   \put(35, 46){$A_q$}
    \put(57,45){$C_q$}
    \put(36.5,13){$B_q$}
    \put(56,2){$D_q$}
    \put(44, 35){\small{$v_S$}}
    \put(43,23){\small{$w_r$}}
    \put(50,24.5){\footnotesize{$w_{r-1}$}}
    \put(43,10.5){$G_q$}
    \put(68, 45){$A_r$}
    \put(89,45){$C_r$}
    \put(89,13){$B_r$}
    \put(69,2){$D_r$}
    \put(75, 35.5){\small{$v_S$}}
    \put(74.5,23){\small{$w_r$}}
    \put(81.5,22){\footnotesize{$w_{r+1}$}}
    \put(81,9){$E_r$}
    \end{overpic}
    \caption{The $p$, $q$ and $r$ half-twists where $q$ is even. These are arranged ready for the ends to be connected to form a pretzel knot.}
    \label{fig:pretzel-pqr}
    \end{figure}
\end{center}

In order to complete the pretzel knot, we join the ends of the strands of half-twists in a cyclic order (as shown in Figure~\ref{fig:pretzel-fold}). Namely, we need to join:
\begin{compactenum}
\item ends $C_p$ and $A_q$, and ends $B_p$ and $B_q$ of the $p$ and $q$ half-twists, 
\item ends $C_q$ and $A_r$, and ends $D_q$ and $D_r$ of the $q$ and $r$ half-twists, 
\item ends $A_p$ and $C_r$, and ends $D_p$ and $B_r$ of the $p$ and $r$ half-twists.
\end{compactenum}

{\bf Join 1:} Joining  ends $C_p$ and $A_q$, and ends $B_p$ and $B_q$ of the $p$ and $q$ half-twists.
\\
Step 1:  First, rotate the strand of $q$ half-twists by $\pi$ radians along the axis of the accordion folds (so they now lie upside down).  Second, move them so they lie to the left of the $p$ half-twists. Finally, place the $q$ half-twists on top of the $p$ half-twists so that the odd lines $o_p=o_q=o$ coincide. Similarly, the start vertices $v_S$ coincide and the vertices $w_r$ coincides. This is shown on the right in Figure~\ref{fig:pretzel-s-qp}.  The ends we need to join together ($C_p$ to  $A_q$,  and $B_p$ to $B_q$) all lie on the right side of the overlapping half-twists.

\begin{center}
    \begin{figure}[htpb]
    \begin{overpic}{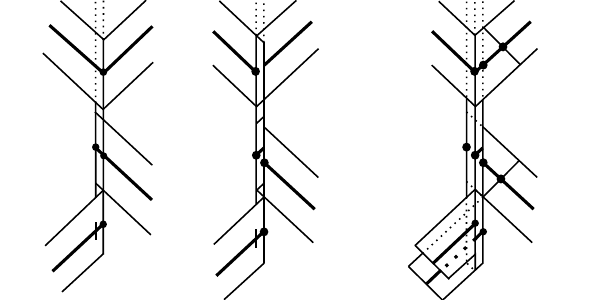}
    \put(13,50.5){$e_q$}
    \put(17,50.5){$o_q$}
    \put(3.5, 46){$C_q$}
    \put(26,46){$A_q$}
    \put(25.5,15){$B_q$}
    \put(5,3){$D_q$}
    \put(18, 36.5){\small{$v_S$}}
    \put(18,24){\small{$w_r$}}
    \put(8,26.5){\small{$w_{r-1}$}}
    \put(18,11.5){\small{$E_q$}}  
    \put(44,50){$e_p$}
    \put(39,50){$o_p$}
   \put(32, 46){$A_p$}
    \put(52,47){$C_p$}
    \put(53,14){$B_p$}
    \put(32,2){$D_p$}
    \put(39, 36){\small{$v_S$}}
    \put(38,24){\small{$w_r$}}
    \put(45,23){\footnotesize{$w_{r+1}$}}
    \put(44.5,10){\small{$E_p$}}
    \put(74,49.5){$e_q$}
    \put(78,50){$o$}
    \put(81,49.5){$e_p$}
    \put(89, 46){$A_q=C_p$}
    \put(62,46){$C_q,A_p$}
    \put(89,13){$B_p=B_q$}
    \put(67,0){$D_p$}
    \put(66,12){$D_q$}
    \put(75, 36.5){\small{$v_S$}}
    \put(81,36){\small{$Y$}}
    \put(82.75, 43.5){\small{$Z$}}
    \put(70,26.5){\small{$w_{r-1}$}}
    \put(81.5,22.5){\footnotesize{$w_{r+1}$}}
    \put(81.5, 16.5){\small{$X$}}
    \put(81,10){\small{$E_p$}}
    \end{overpic}
    \caption{On the right, the flipped $q$ half-twists have been carefully positioned on top of the $p$ half-twists.}
    \label{fig:pretzel-s-qp}
    \end{figure}
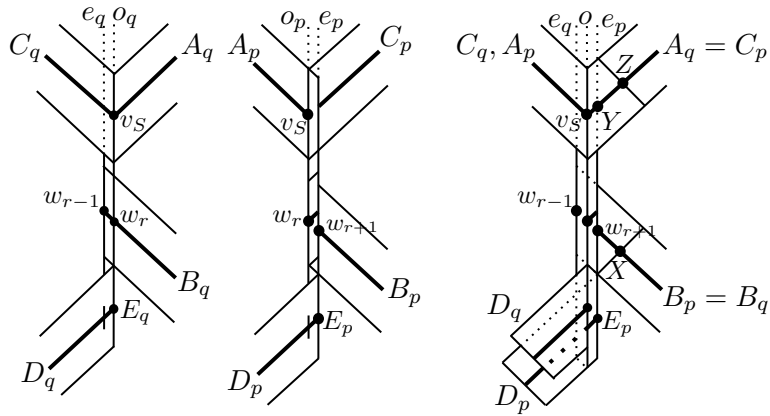
\end{center}
\noindent
Step 2: Since ends $A_q$ to $C_p$ coincide, we join them along a fold line with fold angle 0, one end of which touches the even fold line $e_p$.  This is shown in Figure~\ref{fig:pretzel-s-qp} right, with the fold line passing through point $Z$ on the knot diagram. Point $Y$ is the point near $v_S$ where the knot diagram intersects with even fold line $e_p$, thus $d(v_S,Y)=d$.  From Lemma~\ref{lem:foldline0} we have $d(Y,Z)=\frac{1}{2\sqrt{3}}$. Thus joining $A_q$ to $C_p$ adds $2d(v_S,Z)=2d + \frac{1}{\sqrt{3}}$ units to the folded ribbonlength. 
\\
Step 3: 
Note that ends  $B_p$ and $B_q$ coincide since they are parallel and end $B_q$ passes through vertex $w_r$ and end $B_p$ starts at vertex $w_r$. We join these ends along a fold line with fold angle 0, one end of which touches the even fold line $e_p$.  This is shown in Figure~\ref{fig:pretzel-s-qp} right, with the fold line passing through point $X$ on the knot diagram near $w_{r+1}$. From Lemma~\ref{lem:foldline0}, we know $d(w_{r+1}, X)=\frac{1}{2\sqrt{3}}$. Thus, end $B_q$ adds $d(w_{r-1},w_{r+1})+ d(w_{r+1}, X) = 2d+\frac{1}{2\sqrt{3}}$, and end $B_p$ adds $d(w_r,w_{r+1})+ d(w_{r+1}, X) = d+\frac{1}{2\sqrt{3}}$ to the folded ribbonlength. 
\\
Step 4: Altogether we see, from Steps 2 and 3, that Join 1 adds a total of $5d+\frac{2}{\sqrt{3}}$ units of folded ribbonlength.

{\bf Join 2:} Joining ends $C_q$ and $A_r$, and ends $D_q$ and $D_r$ of the $q$ and $r$ half-twists.
\\
Step 1: We now join the flipped $q$ half-twists to the $r$ half-twists. To do this, we place the $r$ half-twists on top of the flipped $q$ half-twists so that the odd fold lines coincide ($o_q=o_r=o$), and the start vertices $v_S$ coincide. This means the vertices $w_r$ of the $q$ and $r$ half-twists also coincide. This is shown on the right in Figure~\ref{fig:pretzel-s-qr}. This time, the ends we need to join together ($C_q$ to $A_r$ and $D_q$ to $D_r$) all lie on the left side of the overlapping twists.

\begin{center}
    \begin{figure}[htpb]
    \begin{overpic}{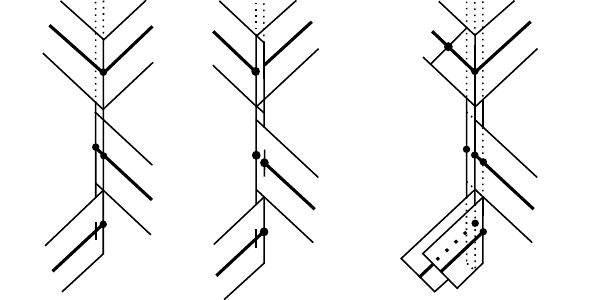}
    \put(13,51){$e_q$}
    \put(17,51){$o_q$}
    \put(3.5, 46){$C_q$}
    \put(26,46){$A_q$}
    \put(25,15){$B_q$}
    \put(4.5,3){$D_q$}
    \put(18.25, 36.5){\small{$v_S$}}
    \put(18,24){\small{$w_r$}}
    \put(8,27){\small{$w_{r-1}$}}
    \put(18,12){\small{$E_q$}} 
    \put(44,50){$e_r$}
    \put(39.5,50){$o_r$}
   \put(32, 46){$A_r$}
    \put(52,47){$C_r$}
    \put(53,14){$B_r$}
    \put(32,2){$D_r$}
    \put(39, 36){\small{$v_S$}}
    \put(38,24){\small{$w_r$}}
    \put(45,23){\footnotesize{$w_{r+1}$}}
    \put(44.5,10){\small{$E_r$}}
    \put(74,50){$e_q$}
    \put(78,50.5){$o$}
    \put(81,50){$e_r$}
    \put(89, 46){$C_r$}
    \put(61,46){$C_q=A_r$}
    \put(89,13){$B_r, B_q$}
    \put(67,1){$D_q$}
    \put(78,1){$D_r$}
    \put(80,36.75){\small{$v_S$}}
    \put(73.5, 38.5){\footnotesize{$Z$}}
    \put(70,27){\small{$w_{r-1}$}}
    \put(81,23.5){\footnotesize{$w_{r+1}$}}
    \put(80.5,13){\footnotesize{$E_q$}}
    \put(81,10){\footnotesize{$E_r$}}
    \end{overpic}
    \caption{On the right, the $r$ half-twists are positioned over the flipped $q$ half-twists.}
    \label{fig:pretzel-s-qr}
    \end{figure}
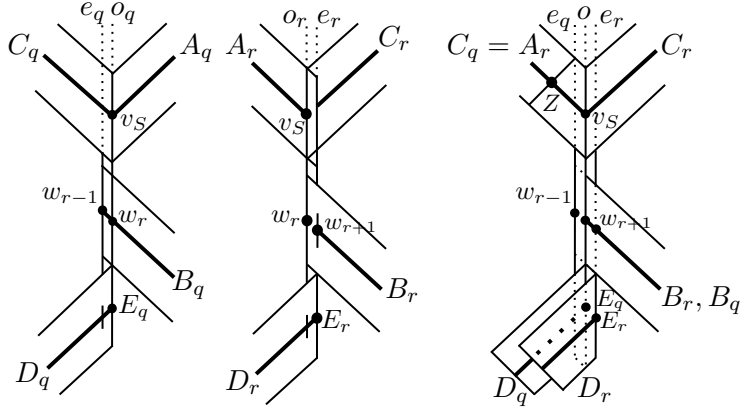
\end{center}
\noindent
Step 2:  Ends $C_q$ and $A_r$ coincide, so we join them along a fold line of fold angle 0, one end of touches the even
fold line $e_q$. This is shown on the right in Figure~\ref{fig:pretzel-s-qr}, with the fold line passing through point Z on the knot
diagram. Following similar reasoning to Join 1, joining $C_q$ to $A_r$ adds $2d(v_S,Z) = 2d+\frac{1}{\sqrt{3}}$. 
units to the folded ribbonlength.

\begin{center}
    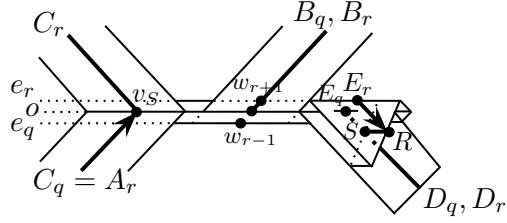
\begin{figure}[htpb]
    \begin{overpic}{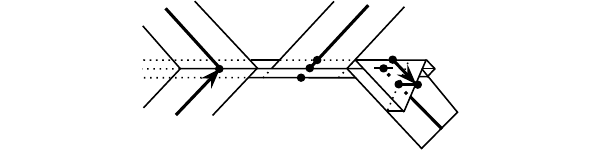}
    \put(23,24){$C_r$}
    \put(23,3.5){$C_q=A_r$}
    \put(57,25.5){$B_q,B_r$}
    \put(74,1){$D_q, D_r$}
    \put(63.5,16.25){\small{$E_r$}}
     \put(60,15){\footnotesize{$E_q$}}
    \put(20,11){$e_q$}
    \put(22,13){$o$}
    \put(20,16){$e_r$}
    \put(36, 15){\small{$v_S$}}
    \put(48,9.5){\footnotesize{$w_{r-1}$}}
    \put(49,16.5){\footnotesize{$w_{r+1}$}}
    \put(70,8.5){\small{$R$}}
    \put(63.5,10){\footnotesize{$S$}}
   \end{overpic}
    \caption{The overlapping $q$ and $r$ half-twists have been rotated to show how end $D_r$ is folded to join end $D_q$.}
    \label{fig:pretzel-ejoin3}
    \end{figure}
\end{center}
\noindent
Step 3: To join end $D_r$ to $D_q$, we need a similar folding pattern to the one in Lemma~\ref{lem:triangle}. The details are shown in Figure~\ref{fig:pretzel-ejoin3}. Start by taking end $D_r$ from vertex $E_r$, then make a right underfold at vertex $R$ with fold angle $\pi/3$. Here, vertex $R$ is chosen so the top side-edge of the end $D_r$ coincides with the even fold line $e_r$. This means, by Lemma~\ref{lem:triangle},  that $d(E_r,R) = \frac{1}{\sqrt{3}}$. We then travel distance $d$ to vertex $S$ and make a left underfold with fold angle $\pi/3$ in end $D_q$ there. By construction, ends $D_r$ and $D_q$ coincide. We these together with a fold of fold angle 0, one end of which touches the fold line through  vertex $S$. To simplify Figure~\ref{fig:pretzel-s-qr} , we have not shown the final ``join'' fold line. The geometry of this final fold line matches Lemma~\ref{lem:foldline0}.  Hence, we find end $D_r$ adds $\frac{1}{\sqrt{3}}+d+\frac{1}{2\sqrt{3}}$ to the ribbonlength, and end $D_q$ adds $\frac{1}{\sqrt{3}}-d +\frac{1}{2\sqrt{3}}$ to the ribbonlength.
\\
Step 4:  Altogether we see, from Steps 2 and 3, that Join 2 adds a total of $\frac{4}{\sqrt{3}}+2d$ units of folded ribbonlength.

{\bf Join 3:} Joining ends $A_p$ and $C_r$, and ends $D_p$ and $B_r$ of the $p$ and $r$ half-twists. The geometry is identical to Join 3 from Construction~\ref{const:pretzel-odd}. Thus Join 3 adds $3d+\frac{7}{\sqrt{3}}$ to the folded ribbonlength. 
\qed
\end{const}

\begin{theorem}\label{thm:pretzel1} Any $P(p,q,r)$ pretzel link type $L$, contains a folded ribbon link $L_w$, such that for any $\epsilon>0$, the folded ribbonlength is 
$$\Rib(L_w)= \begin{cases} \frac{55}{\sqrt{3}} +\epsilon & \text{ when $p$, $q$, $r$ have the same parity},
\\ \frac{51}{\sqrt{3}} +\epsilon\leq 29.45 & \text{ when one of $p$, $q$, $r$ has the opposite parity to the others}.
\end{cases}
$$
\end{theorem}

\begin{proof}  We find the folded ribbonlength of $P(p,q,r)=L$ pretzel link from Construction~\ref{const:pretzel-odd}. Here we assumed $p$, $q$ and $r$ have the same parity and $|p|\leq |q|\leq |r|$. By Proposition~\ref{prop:twists}, the folded ribbonlength of the three strands of half-twists is $3(\frac{12}{\sqrt{3}}+ 2|r|d-d)$.  From Construction~\ref{const:pretzel-odd} we know that
\begin{compactenum}
\item joining $A_q$ to $C_p$ and $D_q$ to $B_p$ adds $\frac{6}{\sqrt{3}}+2d$,
\item joining $C_q$ to $A_r$ and $D_r$ to $B_q$ also adds $\frac{6}{\sqrt{3}}+2d$,
\item joining $A_p$ to $C_r$ and $D_p$ to $B_r$ adds $\frac{7}{\sqrt{3}}+3d$. 
\end{compactenum}
Altogether we find that the folded ribbonlength of the folded ribbon pretzel link $L_w$ is 
\begin{align*}\Rib(L_w) & = \frac{36}{\sqrt{3}} +6|r|d - 3d + \frac{6}{\sqrt{3}}+ 2d+ \frac{6}{\sqrt{3}}+ 2d +\frac{7}{\sqrt{3}} +3d
\\ & = \frac{55}{\sqrt{3}} + d(6|r| +4)  \leq 31.755    	
\end{align*}  
We find the folded ribbonlength of $P(p,q,r)=L$ pretzel link from Construction~\ref{const:pretzel-parity}. Here, we assumed one of $p$, $q$, and $r$ has the opposite parity to the other ones. Without loss of generality, we made other assumptions, for example that $p$, $r$ are odd and $q$ is even and $|p|\leq |r|$ and $|q|\leq |r|$. By Proposition~\ref{prop:twists}, the folded ribbonlength of the three strands of half-twists is $2(\frac{12}{\sqrt{3}}+ 2|r|d -d) + \frac{12}{\sqrt{3}}+ 2|r|d -3d$. 
From Construction~\ref{const:pretzel-parity} we know that
\begin{compactenum}
\item joining $A_q$ to $C_p$ and $B_q$ to $B_p$ adds $5d+\frac{2}{\sqrt{3}}$,
\item joining $C_q$ to $A_r$ and $D_r$ to $D_q$ also adds $2d+\frac{4}{\sqrt{3}}$,
\item joining $A_p$ to $C_r$ and $D_p$ to $B_r$ adds $3d + \frac{7}{\sqrt{3}}$. 
\end{compactenum}
Altogether we find that the folded ribbonlength of the folded ribbon pretzel link $L_w$ is 
\begin{align*}\Rib(L_w) & = \frac{36}{\sqrt{3}} +6|r|d - 5d +5d+\frac{2}{\sqrt{3}} +2d+\frac{4}{\sqrt{3}} +3d+ \frac{7}{\sqrt{3}}
\\ & = \frac{49}{\sqrt{3}} + d(6|r| +5)  \leq  28.291  	
\end{align*}
We achieve statement of the theorem by simply picking $d$ small enough after choosing $\epsilon$.
\end{proof}

Since Theorem~\ref{thm:pretzel1} holds for all $\epsilon>0$, Theorem~\ref{thm:pretzel} immediately follows. That is, the infimal ribbonlength of any $P(p,q,r)$ pretzel link is
$$\Rib([P(p,q,r)])\leq \begin{cases} \frac{55}{\sqrt{3}}  \leq 31.755   & \text{ when $p$, $q$, $r$ have the same parity},
\\  \frac{49}{\sqrt{3}}  \leq  28.291 & \text{ when one of $p$, $q$, $r$ has the opposite parity to the others}.
\end{cases}
$$

What about $n$-strand pretzel links $P(p_1,p_2, \dots, p_n)$? There are more cases to be thought through here, depending on the parity of the $p_i$ and whether $n$ is odd or even. However, it is relatively straightforward to find some kind upper bound on folded ribbonlength, even though it is not the best possible. 

\begin{theorem}\label{thm:pretzel-m} Any $n$ strand pretzel link type $P(p_1,p_2,\dots p_n)=L$, contains a folded ribbon link $L_w$, such that for any $\epsilon>0$, the folded ribbonlength is 
$$ \Rib(L_w) \leq \frac{18n+1}{\sqrt{3}} + \epsilon.$$
\end{theorem}

\begin{proof} Figure~\ref{fig:pretzel-box} shows how the $n$ strands of half-twists of the pretzel link $L$ are connected in a cyclic manner. We can extend Constructions~\ref{const:pretzel-odd} and~\ref{const:pretzel-parity} to $L$ by continuing to flip over every other strand of half-twists. We also arrange for each successive strand of half-twists so that the odd fold lines $o_{p_i}$ coincide and so that all start vertices $v_S$ coincide. We know by Proposition~\ref{prop:twists} that the folded ribbonlength of the $n$ strands of half-twists combined is $\leq n(\frac{12}{\sqrt{3}}+ 2|k|d -d)$, where $k=\max\{|p_1|, |p_2|, \dots |p_n|\}$. Next, consider two consecutive strands of half-twists, the $p_i$ and $p_{i+1}$ half-twists (for $i=1, 2, \dots n-1$). As shown in Theorem~\ref{thm:pretzel1}, when we join the right ends of the strand of $p_i$ half-twists to the left ends of the strand of $p_{i+1}$ twists, we use at most $\frac{6}{\sqrt{3}}+2d$ units of folded ribbonlength. Assuming $n$ is odd, when we join the left ends of the $p_1$ half-twists to the right ends of the $p_n$ half-twists we use at most $3d + \frac{7}{\sqrt{3}}$ units of folded ribbonlength.  If $n$ is even, a moment's thought shows we use far less folded ribbonlength than in the $n$ is odd case. Altogether we find
\begin{align*} \Rib(L_w) &\leq n(\frac{12}{\sqrt{3}}+ 2|k|d -d) + (n-1)(\frac{6}{\sqrt{3}}+2d) + 3d + \frac{7}{\sqrt{3}}
\\ & = \frac{18n+1}{\sqrt{3}} + d(2n|k|-n+2n-2+3)
\\ & =  \frac{18n+1}{\sqrt{3}} + d(2n|k| + n +1).
\end{align*}
We achieve statement of the theorem by simply picking $d$ small enough after choosing $\epsilon$.
\end{proof}
Since $\epsilon>0$ is arbitrary, Theorem~\ref{thm:pretzel-mi} immediately follows. That is, the infimal folded ribbonlength of any $n$-strand pretzel link type $P(p_1,p_2,\dots p_n)$, is 
$ \Rib([P(p_1,p_2,\dots p_n)]) \leq \frac{18n+1}{\sqrt{3}}.$

How do these results compare with the previous bounds? We know from \cite{CDPP} that 3-strand pretzel links $P(p,q,r)$ have infimal folded ribbonlength $\Rib([P(p,q,r)])\leq (|p|+|q|+|r|) + 6$. This bound beats Theorem~\ref{thm:pretzel} for all $p,q, r$ where 
$$|p|+|q|+|r|\leq \begin{cases} 25 & \text{ when $p$, $q$, $r$ have the same parity},
\\ 22 & \text{ when one of $p$, $q$, $r$ has the opposite parity to the others}.
\end{cases}$$
 We know that the crossing number $\Cr(P(p,q,r))\leq |p|+|q|+|r|$, thus the uniform bound of Theorem~\ref{thm:pretzel} only comes into its own once the crossing number is reasonably large. Similarly for $n$-strand pretzel link $L=P(p_1,p_2,\dots,p_n)$, the upper bound of  $\Rib([L])\leq(\sum_{i=1}^n|p_i|)+2n$ from \cite{CDPP} holds for small crossing knots, that is when $\sum_{i=1}^n|p_i| \leq \frac{18n+1}{\sqrt{3}} -2n$.


\section{Table of knots with their folded ribbonlength}\label{sect:table}
In this section we create several tables which give with the best known upper bounds on folded ribbonlength for each knot type with crossing number $\leq 9$. Before we do this, we repeat the upper bounds that we use from Section~\ref{sect:intro}.

\begin{enumerate}
\item Any $(2,q)$-torus knot has $\Rib([T(2,q)])\leq q+3$ from \cite{CDPP}. \item Any $(p,q)$-torus link $L$ has $\Rib([L])\leq 2p$ where $p\geq q\geq 2$ from \cite{Den-FRF}. 
\item Any twist knot $T_n$, with $n$ half-twists has $\Rib([T_n])\leq n+6$ from \cite{CDPP}.
 \item Any $(p,q,r)$-pretzel link $L$ has $\Rib([L])\leq (|p|+|q|+|r|) + 6$ from \cite{CDPP}.
\item Any $n$-strand pretzel link $L$ constructed from $(p_1,p_2,\dots,p_n)$ half-twists, has $\Rib([L])\leq(\sum_{i=1}^n|p_i|)+2n$ from \cite{CDPP}.
\item Any $2$-bridge knot with crossing number $\Cr(K)$ has $\Rib([K])\leq  2\Cr(K)+2$ from \cite{KNY-2Bridge}.
\item Any knot $K$ has $\Rib([K])\leq 2.5\Cr(K)+1$ from \cite{KNY-Lin}.
\end{enumerate}

Before moving to the tables, here are a few remarks. Firstly, the 2-bridge knot bound in (6) is always smaller than the bound for any knot given in (7). Secondly, for small crossing knots, we mostly use the bounds for particular families of knots. Thirdly, the uniform upper bounds on the folded ribbonlength for the $(2,q)$ torus and twist knots\footnote{ For $(2,q)$ torus knots $\Rib([T(2,q)]) \leq 8\sqrt{3}\leq 13.86$ is only relevant for odd $q \geq 11$, that is crossing number $11$ and up. For twist knots we $\Rib([T_n])\leq  9\sqrt{3}+2 \leq 17.59$ for $n$ odd and $\Rib([T_n])\leq  8\sqrt{3} +2 \leq  15.86$ for $n$ even. These are relevant once $n \geq 10$, equivalently crossing number 12 and higher. } given in Section~\ref{sect:intro} are only relevant when the crossing number is 11 and higher. As discussed in the Section~\ref{sect:ribbon}, the uniform bounds on the folded pretzel links aren't relevant for the small crossing knots we consider here. Finally, when using the 2-bridge knot bound, we use the notation for 2-bridge knots found in KnotInfo \cite{knotinfo}. That is, any 2-bridge knot can be reconstructed from the fraction  $p/q$ (where  $0<q<p$) by finding a continued fraction expansion of $p/q$ (see for instance \cite{Crom}). We denote a 2-bridge knot, which is also knows as a {\em rational knot}, with the notation $R(p,q)$.

In Table~\ref{tbl:6}, we have listed the upper bounds on the folded ribbonlength of knots with crossing number $\leq 6$. In the ``Notes'' column we have shown which formula we have used to find the upper bounds. These results have already appeared in \cite{CDPP}, but we include them here for completeness. It may well be possible to reduce the folded ribbonlength for the $6_2$ and $6_3$ knots.

Table~\ref{tbl:7-8} contains the upper bounds on folded ribbonlength of knots with crossing number 7 and~8. Only three knots use the universal upper bound formula: $\Rib([K])\leq  2.5\cdot \Cr(K)+1$. Namely, the three non-rational and non-pretzel 8 crossing knots: $8_{16}, 8_{17}$, and $8_{18}$.

The 9 crossing knots are found in Table~\ref{tbl:9}. There are 49 knots with crossing number 9. Of these, 18 knots are neither pretzel nor rational knots and use the universal upper bound of $\Rib([K])\leq  2.5\cdot 9+1=23.5$.  These are $9_{22}$. $9_{25}$, $9_{29}$, $9_{30}$, $9_{32}$, $9_{33}$, $9_{34}$, $9_{36}$, $9_{38}$, $9_{39}$, $9_{40}$, $9_{41}$, $9_{42}$, $9_{43}$, $9_{44}$, $9_{45}$, $9_{47}$, and $9_{49}$. To save room, we have listed them in Table~\ref{tbl:9} under ``Others''.

When we look at 10 crossing knots, the situation is very simple.  No one has yet extended pretzel notation to knots with $\geq 10$ crossings (or to links). We know knots 1 through 45 are rational and knots 46 through 165 are not rational. We have also separated out the single twist knot and two torus knots. We have summarized the information in Table~\ref{tbl:10}.

Other knots and links with known folded ribbonlength bounds can be found in Table~\ref{tbl:other}. The link nomenclature follows LinkInfo~\cite{linkinfo}. Some of this data previously appeared in \cite{CDPP}.

\begin{center}
\begin{table}[htbp]
\caption{Folded ribbonlength of knots with crossing number $\leq 6$.}
\begin{tabular}{ | c | c | c | l |}
\hline
Knot&  $\Rib([K])\leq$ ??? & Notes
\\ \hline
$0_1$ &  0 & 
\\ \hline
$3_1$ &  6 &  $\Rib([T(2,3)])\leq 3+3= 6$
\\ \hline
$4_1$ & 8 & $\Rib([T_2])\leq 2+6 = 8$
\\ \hline
$5_1$ & 8 & $\Rib([T(2,5))]\leq 5+3= 8$
\\ \hline
$5_2$&  9 & $\Rib([T_3])\leq 3+6 = 9$
\\ \hline
$6_1$&  10 & $\Rib([T_4])\leq 4+6 = 10$
\\ \hline
$6_2$& 12 & $\Rib([P(1,2,3)])\leq 1+2+3+6 = 12$
\\ \hline
$6_3$&  14 & $\Rib([R(13,5)])\leq 2\cdot6+2= 14$
\\ \hline
\end{tabular}
 \label{tbl:6}
\end{table}
\end{center}

\begin{center}
\begin{table} [htbp]
\caption{Folded ribbonlength of knots with crossing number $7$ and $8$.} 
\begin{tabular}{ | c | c | c | p{7cm} |} 
\hline
Knot& $\Rib([K])\leq$ ??? & Notes 
\\ \hline
$7_1$ & 10 &  $\Rib([T(2,7)])\leq 7+3= 10$
\\ \hline
$7_2$ & 11 & $\Rib([T_5])\leq 5+6 = 11$
\\ \hline
$7_3$  & 16 & $\Rib([R(13,3)])\leq 2\cdot 7+2=16$
\\ \hline
$7_4$& 13 & $\Rib([P(-3,-1,-3)])\leq 3+1+3+6=13$
\\ \hline
$7_5$& 15 & $\Rib([P(2,1,1,3])\leq 2+1 +1+3+8 = 15$
\\ \hline
$7_6$& 16  & $\Rib([R(19,7)])\leq 2\cdot 7+2=16$
\\ \hline
$7_7$& 16 & $\Rib([R(21,8)])\leq 2\cdot 7+2=16$
\\ \hline
\hline
$8_1$ &  12  & $\Rib([T_6])\leq 6+6=12$
\\ \hline
$8_2$ & 14 & $\Rib([P(1,2,5) ])\leq  1+2+5+6=14$
\\ \hline
$8_3$ & 18 & $\Rib([ R(17,4)]) \leq 2\cdot 8 +2 =18$
\\ \hline
$8_4$ & 14 & $\Rib([P(-1,-3,-4) ])\leq 1+3+4+6=14 $
\\ \hline
$8_5$ & 14 & $\Rib([P(2,3,3) ])\leq 2+3+3+6=14 $
\\ \hline
$8_6$ & 18  & $\Rib([ R(23,7)]) \leq 2\cdot 8 +2 =18$
\\ \hline
$8_7$ & 17 & $\Rib([P(-4,-1,3,-1) ])\leq 4+1+3+1+8= 17$
\\ \hline
$8_8$ &  18 & $\Rib([ R(25,9)]) \leq 2\cdot 8 +2 =18$
\\ \hline
$8_9$ &  17  & $\Rib([ P(3,1,-4,1)])\leq  3+1+4+1+8=17$
\\ \hline
$8_{10}$ & 17 & $\Rib([P(-3,-2,3,-1) ])\leq 3+2+3+1+8 = 17$
\\ \hline
$8_{11}$ &  18  & $\Rib([ R(27,10)]) \leq 2\cdot 8 +2 =18$
\\ \hline
$8_{12}$ & 18 & $\Rib([ R(29,12)]) \leq 2\cdot 8 +2 =18$
\\ \hline
$8_{13}$ & 18 &$\Rib([ R(29,3)]) \leq 2\cdot 8 +2 =18$
\\ \hline
$8_{14}$ & 18 & $\Rib([ R(31,12)]) \leq 2\cdot 8 +2 =18$
\\ \hline
$8_{15}$ & 20  & $\Rib([P(3,-1,-2,-1,3) ])\leq 3+1+2+1+3+10= 20$
\\ \hline
$8_{16}$ & 21 & $\Rib([ 8_{16}])\leq  2.5\cdot 8+1=21$
\\ \hline
$8_{17}$  & 21 & $\Rib([ 8_{17}])\leq  2.5\cdot 8+1=21$
\\ \hline
$8_{18}$  & 21 & $\Rib([ 8_{18}])\leq  2.5\cdot 8+1=21$
\\ \hline
$8_{19}$ & 8 & $\Rib([T(4,3)])\leq 2\cdot 4=8 $
\\ \hline
$8_{20}$ & 17 & $\Rib([P(3,-2,-3,1) ])\leq  3+2+3+1+8=17 $
\\ \hline
$8_{21}$ & 17  & $\Rib([P(3,3,-1,-2) ])\leq 3+3+1+2+8= 17 $
\\ \hline
\end{tabular}
 \label{tbl:7-8}
\end{table}
\end{center}

\begin{center}
\begin{table}[htbp]
\caption{Folded ribbonlength of knots with crossing number $9$.}
\begin{tabular}{ | c | c | c | p{6cm} |} 
\hline
Knot & $\Rib([K])\leq$ ??? & Notes 
\\ \hline
$9_1$  & 12 & $\Rib([T(2,9) ])\leq 9+3=12 $
\\ \hline
$9_2$  &  13 & $\Rib([T_7 ])\leq 7+6 =13 $
\\ \hline
$9_3$  & 16 & $\Rib([ P(1,-4,5)])\leq  1+4+5+6=16$
\\ \hline
$9_4$ & 16 & $\Rib([ P(-1,5,-4)])\leq 1+5+4+6=16 $
\\ \hline
$9_5$  & 15 & $\Rib([P(-1,-3,-5) ])\leq 1+3+5+6= 15 $
\\ \hline
$9_6$  & 19& $\Rib([P(-1,-1,3,6) ])\leq 1+1+3+6+8 = 19 $
\\ \hline
$9_7$ &  20 & $\Rib([ R(29,9)]) \leq 2\cdot 9 +2 =20$
\\ \hline
$9_8$ & 20  & $\Rib([ R(31,11)]) \leq 2\cdot 9 +2 =20$
\\ \hline
$9_9$   & 17 & $\Rib([P(4,1,1,3) ])\leq 4+1+1+3+8 =17$
\\ \hline
$9_{10}$ & 19& $\Rib([P(-3,-1,-1,-1,-3)  ])\leq 3+1+1+1+3 + 10= 19 $
\\ \hline
$9_{11}$  &20 & $\Rib([ R(33,14)]) \leq 2\cdot 9 +2 =20$
\\ \hline
$9_{12}$ & 20 & $\Rib([ R(35,13)]) \leq 2\cdot 9 +2 =20$
\\ \hline
$9_{13}$  &20 & $\Rib([ R(37,11)]) \leq 2\cdot 9 +2 =20$
\\ \hline
$9_{14}$ & 20 & $\Rib([ R(37,8)]) \leq 2\cdot 9 +2 =20$
\\ \hline
$9_{15}$ & 20& $\Rib([ R(39,16)]) \leq 2\cdot 9 +2 =20$
\\ \hline
$9_{16}$  & 17 & $\Rib([ P(2,3,1,3)])\leq 2+3+1+3+8= 17 $
\\ \hline
$9_{17}$ & 20 & $\Rib([ R(39,14)]) \leq 2\cdot 9 +2 =20$
\\ \hline
$9_{18}$ & 20 & $\Rib([ R(41,17)]) \leq 2\cdot 9 +2 =20$
\\ \hline
$9_{19}$ & 20 & $\Rib([ R(41,16)]) \leq 2\cdot 9 +2 =20$
\\ \hline
$9_{20}$ & 20 & $\Rib([ R(41,11)]) \leq 2\cdot 9 +2 =20$
\\ \hline
$9_{21}$ & 20 & $\Rib([ R(43,12)]) \leq 2\cdot 9 +2 =20$
\\ \hline
$9_{23}$ & 20 & $\Rib([ R(45,19)]) \leq 2\cdot 9 +2 =20$
\\ \hline
$9_{24}$ & 20 & $\Rib([P(-3,-1,3,-1,-2) ])\leq 3+1+3+1+2+10= 20 $
\\ \hline
$9_{26}$ & 20 & $\Rib([ R(47,18)]) \leq 2\cdot 9 +2 =20$
\\ \hline
$9_{27}$ & 20 & $\Rib([ R(49,19)]) \leq 2\cdot 9 +2 =20$
\\ \hline
$9_{28}$ & 23 & $\Rib([ P(-1,-1,3,-2,-1,3)])\leq 1+1+3+2+1+3+12= 23 $
\\ \hline
$9_{31}$ & 20 & $\Rib([ R(55,21)]) \leq 2\cdot 9 +2 =20$
\\ \hline
$9_{35}$ &  15 & $\Rib([ P(-3,-3,-3)])\leq 3+3+3+6=15 $
\\ \hline
$9_{37}$  & 21& $\Rib([P(-1,3,-3,3,-1) ])\leq 1+3+3+3+1+10+21 $
\\ \hline
$9_{46}$  &15 & $\Rib([P(3,3,-3 ])\leq 3+3+3+6=15 $
\\ \hline
$9_{48}$  &21 & $\Rib([P(3,-1,3,-1,3) ])\leq 3+1+3+1+3+10 = 21 $
\\ \hline
Others & 23.5  & $\Rib([K])\leq  2.5\cdot 9+1=23.5$
\\ \hline
\end{tabular}
  \label{tbl:9}
\end{table}
\end{center}

\begin{center}
\begin{table}[htbp]
\caption{Folded ribbonlength of knots with crossing number $10$.}
\begin{tabular}{ | c | c | c | l |}
\hline
Knot&  $\Rib([K])\leq$ ??? & Notes
\\ \hline
$10_1$ &  13 &  $\Rib([T(2,3)])\leq 10+3= 13$
\\ \hline
$10_2$ & 14 & $\Rib([T_{8}])\leq 8+6 = 14$
\\ \hline
$10_{124}$ & 10 & $\Rib([T(5,3))]\leq 2\cdot 5=10$
\\ \hline
$10_3$ through $10_{45}$ &  22 & $\Rib([K])\leq 2\cdot 10+2 = 22$
\\ \hline
$10_{46}$ through $10_{165}$ &  26 & $\Rib([K])\leq 2.5\cdot 10 +2= 26$
\\ \hline
\end{tabular}
 \label{tbl:10}
\end{table}
\end{center}

\begin{center}
\begin{table} [htbp]
\caption{Folded ribbonlength of other knots and links.}
\begin{tabular}{ | c | c | c | c |} 
\hline
Knot/link table& Name & $\Rib([K])\leq$ ??? & Notes 
\\ \hline
12n242 & $P(-2,3,7)$ & 18  & $\Rib[P(-2,3,7)]\leq 2+3+7+6=18$
\\ \hline
L2a1 & Hopf link & 4 & $\Rib[\text{Hopf}]\leq 4$
\\ \hline
L4a1 & $(2,4)$-torus & 7 & $\Rib([T(2,4)])\leq 4+3=7$
\\ \hline
L6a3 & $(2,6)$-torus & 9 & $\Rib([T(2,6)])\leq 6+3=9$
\\ \hline
\end{tabular}
 \label{tbl:other}
\end{table}
\end{center}



\section{Acknowledgments}
The folded ribbon project has developed over a number of years, and the author's research has been generously funded by Lenfest grants from Washington \& Lee University (W\&L) over many years. A special thanks go to the author's many W\&L student collaborators on the folded ribbon project, most recently Zhicheng Chen, Kyle Patterson, and Timi Patterson. The author is very grateful for the support and encouragement of Jason Cantarella and Richard Schwartz.  All the figures in this paper were made using Google Draw. 

\bibliography{folded-ribbons}{}
\bibliographystyle{plain}


\end{document}